\documentclass{article}

\usepackage{dsfont}
\usepackage{mathrsfs}
\usepackage{tikz-cd}
\usepackage{mathtools}
\usepackage{bbm}
\usepackage[bbgreekl]{mathbbol}
\usepackage{tikz-cd}
\usepackage{todonotes}
\usepackage{stix}
\usepackage{hyperref}
\usepackage{cleveref}
\usepackage{float}
\usepackage[export]{adjustbox}

\usepackage{amsmath, amsthm, amsfonts}

\DeclareMathSymbol\bbDelta  \mathord{bbold}{"01}

\numberwithin{equation}{section}
\newtheorem{theorem}{Theorem}[section]
\newtheorem{lemma}[theorem]{Lemma}
\newtheorem{proposition}[theorem]{Proposition}

\theoremstyle{definition}
\newtheorem{definition}[theorem]{Definition}
\theoremstyle{remark}
\newtheorem{remark}[theorem]{Remark}

\newtheorem{example}[theorem]{Example}

\newtheorem{notation}[theorem]{Notation}

\renewcommand{\phi}{\varphi}
\renewcommand{\epsilon}{\varepsilon}

\newcommand{\tr}[1]{{#1}^{\#}}

\providecommand{\keywords}[1]
{
  \small	
  \textbf{\textit{Keywords---}} #1
}

\providecommand{\MSCsubj}[1]
{
  \small	
  \textbf{\textit{2020 Mathematics Subject Classification ---}} #1
}

\newcommand{\catA}{\mathscr{A}}
\newcommand{\catB}{\mathscr{B}}
\newcommand{\catC}{\mathscr{C}}
\newcommand{\catD}{\mathscr{D}}

\newcommand{\modcatA}{\mathscr{M}}

\newcommand{\we}{\mathscr{W}}
\newcommand{\indexI}{\mathbb{I}}
\newcommand{\indexJ}{\mathbb{J}}
\newcommand{\monadA}{\mathbb{T}}

\newcommand{\monadAtriple}{(T,\mu,\eta)}

\newcommand{\comonadA}{\mathbb{K}}

\newcommand{\comonadAtriple}{(K,\Delta,\epsilon)}

\newcommand{\n}[1]{\underline{#1}}
\newcommand{\wo}[1]{\widehat{#1}}
\newcommand{\poset}[1]{\mathcal{P}(\hspace{0.1em}\underline{\text{#1}}\hspace{0.1em})}
\newcommand{\I}{\mathscr{I}}
\newcommand{\initial}[1]{-\infty_{#1}}
\newcommand{\terminal}[1]{+\infty_{#1}}
\newcommand{\basepoint}[1]{\infty_{#1}}
\newcommand{\monadcat}{\text{Monad}}

\newcommand{\fun}[2]{\operatorname{Fun}(#1 \text{,}#2)}
\newcommand{\End}[1]{\operatorname{End}(#1)}
\newcommand{\Mod}[1]{\text{Mod}_{#1}}

\newcommand{\ob}[1]{\text{Ob}\,#1}
\newcommand{\id}[1]{\text{Id}_{#1}}
\newcommand{\comp}[1]{#1^\text{c}}
\newcommand{\subsetA}{\mathscr{U}}
\newcommand{\subsetB}{\mathscr{V}}

\DeclareMathOperator{\hocolim}{hocolim}
\newcommand{\hocolimunder}[1]{{\hocolim}_{#1}}
\newcommand{\syshocolim}{\mathbb{H}\text{C}}

\DeclareMathOperator{\icofib}{icofib}

\newcommand{\Ran}[2]{\text{Ran}_{#2} {#1}}

\DeclareMathOperator{\Cat}{Cat}
\newcommand{\im}{\text{Im}}
\newcommand{\whisk}{\,{*}\,}

\newcommand{\theintcube}[3]{\chi^{{#1}}_{{#2},{#3}}}

\newcommand{\ifibcube}[2]{\text{T}^{#1}_{#2}}

\newcommand{\monadcube}[2]{\mathbb{T}^{#1}_{#2}}
\newcommand{\mult}{\mu}

\newcommand{\cross}[1]{\operatorname{cr}^{#1}}

\title{Constructing monads from cubical diagrams and homotopy colimits}

\author{Kristine Bauer \footnote{University of Calgary, \href{mailto:bauerk@ucalgary.ca}{bauerk@ucalgary.ca}}
\and Robyn Brooks\footnote{University of Utah, \href{mailto:robyn.brooks@utah.edu}{robyn.brooks@utah.edu}}
\and Kathryn Hess \footnote{\'{E}cole Polytechnique F\'{e}d\'{e}rale de Lausanne, \href{mailto:kathryn.hess@epfl.ch} {kathryn.hess@epfl.ch}}  
\and Brenda Johnson \footnote{Union College, \href{mailto:johnsonb@union.edu}{johnsonb@union.edu}} 
\and Julie Rasmusen\footnote{University of Warwick, \href{mailto:julie.rasmusen@warwick.ac.uk}{julie.rasmusen@warwick.ac.uk}}
\and Bridget Schreiner \footnote{University of Notre Dame, \href{mailto:bschrein@nd.edu}{bschrein@nd.edu}} 
 \\ }

\begin{document}

\maketitle{}

\begin{abstract}
     This paper is the first step in a general program for defining  cocalculus towers of functors via sequences of compatible monads.
  Goodwillie's calculus of homotopy functors  inspired many new functor calculi in a wide range of contexts in algebra, homotopy theory and geometric topology.  Recently, the third and fourth authors have developed a general program for constructing generalized calculi from sequences of compatible comonads.  In this paper, we dualize the first step of the Hess-Johnson program, focusing on monads rather than comonads.  We consider categories equipped with an action of the poset category $\poset{n}$, called $\poset{n}$-modules.  We exhibit a functor from $\poset{n}$-modules to the category of monads.  The resulting monads act on categories of functors whose codomain is equipped with a suitable notion of homotopy colimits.  In the final section of the paper, we demonstrate the monads used to construct McCarthy's dual calculus as an example of a monad arising from a $\poset{n}$-module.  This confirms that our dualization of the Hess-Johnson program generalizes McCarthy's dual calculus, and serves as a proof of concept for further development of this program.
\end{abstract}

\noindent \keywords{monads, cubical diagrams, homotopy colimits, module categories, functor calculus}

\noindent \MSCsubj{55P65, 18C15, 16D90}
\section{Introduction}

Goodwillie's calculus of homotopy functors  \cite{G90, G91, G03} has played an important role in the establishment of many significant results and insights in $K$-theory and homotopy theory.  It has also inspired the development of many new functor calculi in a wide range of contexts in algebra, homotopy theory, and geometric topology.   Each of these functor calculi is characterized by the existence of a ``Taylor tower'' of ``degree $n$'' functors that approximate a given functor $F$, in a manner analogous to the Taylor series for functions, where the definition of ``degree $n$'' varies from one flavor of functor calculus to the next. 

The  abelian functor calculus of Johnson and McCarthy \cite{JM04}, designed for functors of abelian categories, has its roots in classical constructions of Dold and Puppe  and of Eilenberg and Mac Lane.  The methods used to define the abelian functor calculus were adapted by Bauer, Johnson and McCarthy in \cite{BJM15} to define the discrete calculus for functors of simplicial model categories.  Under certain hypotheses, the discrete calculus agrees with Goodwillie's calculus for homotopy functors, thereby serving as a bridge  between the algebraic abelian functor calculus and the topological calculus of homotopy functors.

The discrete and abelian functor calculi are further distinguished by the means by which they are constructed.  Underlying each stage of the Taylor tower for both calculi is a comonad that is simultaneously used to define what it means for a functor to be of ``degree $n$'' in that particular calculus and to construct, through a standard bar and homotopy colimit constructions,  a ``degree $n$'' approximation to an arbitrary functor.  The nature of these constructions suggests  that these methods can be generalized, to develop new calculi from sequences of compatible comonads, and dualized, to create a framework for building cocalculi from sequences of compatible monads.  A general program for constructing calculi using this approach, is being carried out by the third and fourth authors of this work \cite{HJ21}.  

On the dual side, McCarthy's dual calculus provides an  example of a cocalculus defined via monads that is dual to the discrete calculus of \cite{BJM15}. Comparisons between the discrete calculus and its dual  led to interesting results about  obstructions  to Taylor towers being the product of their homogeneous layers. These results subsequently led to the classification of various types of functors (see, e.g. \cite{McCarthy2001} \cite{Chaoha} \cite{AC11} \cite{Ching10}).  There are also recent applications of dual calculus to  stable cohomology operations \cite{GlasmanLawson} and operads in the $\infty$-category of $T(n)$-local spectra \cite{Heuts}, further illustrating its utility.

This paper is the first step in the process of developing a general program for defining cocalculi via sequences of compatible monads, dual to the program of Hess and Johnson, which will include McCarthy's dual calculus as an example.   

The comonads underlying the abelian and discrete calculi arise as the iterated homotopy fibers of certain {\it $n$-cubical diagrams}, that is, functors out of poset categories $\poset{n}$ whose objects are subsets of  sets $\n{n}:=\{1,\dots,n\}$ where $n$ can be any natural number.  That one can construct the comonads used in the discrete and abelian functor calculi in this way is a consequence of the fact that each cubical diagram arises from a {\it $\poset{n}$-module}, i.e., from a category equipped with an action of $\poset {n}$, where we consider $\poset{n}$ as a symmetric monoidal category with union as monoidal product. In the case of dual calculus, there are monads underlying the constructions which  arise as  iterated cofibers of  cubical diagrams.  The cubical diagrams involved in dual calculus have different features, which indicate that an interpretation of these diagrams as $\poset{n}$-modules require a different monoidal structure.

The goal of this paper is to show that this process of producing comonads  from $\poset{n}$-modules can be dualized to produce monads from $\poset{n}$-modules, when considering $\poset{n}$ equipped with the symmetric monoidal structure given by intersection instead of union. Our main result, which appears as Theorem \ref{theorem:pnmodstomonads}, establishes the existence of a functor from the category of $\poset{n}$-modules to the category of monads and can be summarized as follows.

\smallskip
\noindent{\bf Theorem \ref{theorem:pnmodstomonads}.} {\it For a category $\modcatA$ with homotopy colimits, there exists a contravariant functor from the category of $\poset{n}$-modules to the category of monads.  For a $\poset{n}$-module on a category $\catA$, the corresponding monad acts on the functor category $\fun{\catA}{\modcatA}$.}
\smallskip

  The monads in the target of this functor are obtained by taking the iterated homotopy cofiber of certain cubical diagrams.  In Proposition \ref{thm:cocrossmonad} we show that the monad obtained from a specific familiar $\poset{n}$-module recovers McCarthy's $n$th cocross effect functor evaluated on the diagonal.  This confirms that our dualization of \cite{HJ21} generalizes the constructions in \cite{McCarthy2001}, as desired.

Given the role of $\poset{n}$-modules as a source of monads, we also provide characterizations of $\poset{n}$-modules in terms of other, perhaps more familiar, categorical constructions in Proposition \ref{prop:equivmodulestructure}. The compatibility condition in item (3) is specified in the full statement of the proposition in section 3.  

\smallskip
\noindent{\bf Proposition \ref{prop:equivmodulestructure}.}
{\it The following structures on a small category $\catA$ are equivalent:
\begin{enumerate}
    \item a $\poset{n}$-module structure on $\catA$, 
    \item a set of $n$ strict idempotent commuting comonads on $\catA$, and  
    \item a set of $n$ strict compatible coreflective subcategories of $\catA$.
\end{enumerate}}

This proposition suggests that it may be possible to first identify a reflective subcategory of a category such as $\fun{\catA}{\modcatA}$ as a candidate for the subcategory of degree $n$ functors, and then to construct the associated $\poset{n}$-modules and monads directly to fit the desired notion of degree $n$.  

Since our method for constructing monads from $\poset{n}$-modules recovers McCarthy's cocross effects, it is natural to ask whether or not it can be extended to produce cotowers of functors recovering the dual calculus cotower and to dualize the functor calculi of \cite{HJ21}.  To do so requires tools for comparing monads arising from $\poset{n}$-modules to those arising from $\poset{m}$-modules, for $n\neq m$.  In Proposition 5.4 we establish conditions that ensure the existence of such comparisons, which can be summarized as follows.  

\smallskip
\noindent{\bf Proposition 5.4} 
 {\it Any strict monoidal functor $g:\poset{m}\to \poset{n}$ (i.e., preserving intersections and the empty set) induces a functor $g^*$ from the category of $~\poset{n}$-modules to that of $\poset{m}$-modules.  

Moreover, for any $\poset{n}$-module structure $\theta$ on $\catA$, there is a natural morphism of monads on $\fun{\catA}{\modcatA}$, from the monad associated to the $\poset{m}$-module $g^\ast(\catA,\theta)$ to the monad associated to the $\poset{n}$-module $(\catA,\theta)$.}
\smallskip

In Theorem 6.3 we show that for a surjection from a set with $n$ elements to a set with $m$ elements, Proposition 5.4 produces a natural transformation from the $m$-th cocross effect (evaluated on the diagonal) to the $n$-th cocross effect (also evaluated on the diagonal), which underlies a morphism of the associated monads.  We aim to use this in the development of a full theory of ``functor cocalculi'' in future work.

\subsection{Organization} The paper is organized as follows.  In \textbf{section 2}, we review the definitions of monads, comonads and morphisms between monads, allowing us to define the category of monads. We further discuss ways in which monads and comonads can arise from adjunctions.  We use \textbf{section 3} to introduce $\poset{n}$-modules and establish the equivalence between $\poset{n}$-modules, strict idempotent comonads, and strict coreflective subcategories.  In \textbf{section 4}, we identify the properties for a choice of model for homotopy colimits required for our main construction, which we construct in \textbf{section 5}.  This construction transforms a $\poset{n}$-module into a monad and yields the functor between the category of $\poset{n}$-modules and the category of monads, which is the subject of Theorem \ref{theorem:pnmodstomonads}.  Finally, in \textbf{section 6}, we review the monads of McCarthy's dual calculus and show that they arise from our framework for constructing monads from cubical diagrams and homotopy colimits. 

\subsection{Acknowledgments}

The authors thank the Hausdorff Research Institute for Mathematics for hosting the fourth Women in Topology workshop, which brought us together for a week during which many of the results in this paper were obtained.  We  thank Jelena Grbi\'c, Ang\'elica Osorno, Vesna Stojanoska and Inna Zakharevich for organizing this workshop and working diligently to ensure that all participants had funding to attend. We also thank the Association for Women in Mathematics for providing travel support for several workshop participants.
  The first author would like to thank NSERC for supporting her research. The second author acknowledges the support of ICERM at Brown University, where part of this work was done. 
 The fourth author was partially supported by the Union College Faculty Research Fund.

\section{Monads and Comonads}

In this section we  recall the  definitions of monads and comonads, as well as morphisms between monads to form the category $\monadcat$.  We also describe two important techniques for creating monads from adjunctions, and provide examples of these techniques that will be of use to us in later sections.  

Throughout, we will need to compose natural transformations and functors to obtain a new natural transformation, as follows. Given a natural transformation $\alpha:F\Rightarrow G$ and functors $H$ and $K$ as displayed in
\begin{center}
    \begin{tikzcd}
        \catA
            \arrow[r, "H"]
    &   \catB
            \arrow[r, bend left=40, "F", ""{name=F}]
            \arrow[r, bend right=40, swap, "G", ""{name=G}]
    &   \catC
            \arrow[r, "K"]
    &   \catD
   \arrow[Rightarrow, "\,\alpha", shorten=2mm, from=F, to=G]
     \end{tikzcd}
\end{center}
we define the \textit{whiskered} composite of $\alpha$ with $H$ and $K$ to be the natural transformation $K\whisk\alpha\whisk H:KFH\Rightarrow KGH$ given by $(K\alpha H)_a=K\alpha_{H(a)}$. 

\begin{definition}\label{def:monad}
    A \emph{monad} on a category $\catA$ is a triple $\monadA=\monadAtriple$, consisting of an endofunctor $T:\catA\to\catA$ together with two natural transformations $\mult: T T \Rightarrow T$ and $\eta:\id{\catA}\Rightarrow T$, called the \textit{multiplication} and \textit{unit} respectively, such that the following two diagrams commute:
   
\begin{center}
        \begin{tikzcd}[row sep=1.8em, column sep=3em]
                T T  T
                    \arrow[r, " T\whisk\mult"]
                    \arrow[d, swap, "\mult\whisk T"]
            &   T T
                    \arrow[d, "\mult"]
            &
            &   T
                    \arrow[rd, swap, "\id{T}"]
                    \arrow[r, "\eta \whisk T"]
            &   T T
                    
                    \arrow[d, "\mult" description]
            &   T
                    \arrow[ld, "\id{T}"]
                    \arrow[l, swap, "T\whisk\eta"]
            \\  T T
                    \arrow[r, swap, "\mult"]
            &   T
            &&& T.
            &
                
        \end{tikzcd}
    \end{center}
    We say that a monad $\monadA=\monadAtriple$ is a \textit{strict idempotent} if $\mult$ is in fact the identity. 

The category of monads, denoted by $\monadcat$, has as its objects pairs $(\catA, \monadAtriple)$, where $\monadAtriple$ is a monad on $\catA$. A morphism $(\catA,(T_\catA,\mu_\catA,\eta_\catA))\to (\catB,(T_\catB,\mu_\catB,\eta_\catB))$ in $\monadcat$ is a pair $(F,\alpha )$, consisting of a functor $F:\catA\rightarrow\catB$ together with a natural transformation $\alpha:T_\catB  F\Rightarrow F T_\catA$ such that the following diagrams commute:
\begin{center}\label{e:monadmorph}
        \begin{tikzcd}[column sep=3em]
                T_\catB T_\catB  F
                    \arrow[r, "T_\catB\whisk\alpha\;"]
                    \arrow[d, swap, "\mu_\catB\whisk F"]
            &   T_\catB F T_\catA
                    \arrow[r, "\alpha \whisk T_\catA"]
            &   F T_\catA T_\catA
                    \arrow[d, "F\whisk\mu_\catA"]
            &  F
                    \arrow[rd, "F \whisk\eta_\catA"]
                    \arrow[d, swap,"\eta_\catB\whisk F"]
            &   
            \\  T_\catB F
                    \arrow[rr, swap, "\alpha"]
            &&  F T_\catA

            & T_\catB F
                \arrow[r, swap, "\alpha"]
            & F T_\catA.
            
        \end{tikzcd}
    \end{center}
     
The category $\monadcat_\catA$ is the full subcategory of $\monadcat$ whose objects are monads on $\catA$. 
\end{definition} 
For ease of notation, we will often omit the multiplication and unit from the notation of a monad, and simply denote it by the endofunctor. 

\begin{remark}\label{rem:monoid}
    A monad $\monadA$ on a category $\catA$ is equivalent to a monoid in the category $\End{\catA}$ of endofunctors on $\catA$, where the monoidal product is given by composition.
\end{remark}

\begin{example}\label{ex:identity}
    The simplest example of a monad on a category $\catA$ is the identity functor $\id{\catA}$, with multiplication and unit given by the identity morphisms on the objects. 
\end{example}

Dually, one can define a comonad on a category $\catA$ as being a monad on $\catA^{op}$, which is equivalent to a comonoid in $\End{\catA}$, and which more explicitly can be defined  as follows. 

\begin{definition}\label{defcomonad}
    A \emph{comonad} on a category $\catA$ is a triple $\comonadA =\comonadAtriple$ consisting of an endofunctor $K:\catA\to\catA$ together with two natural transformations $\Delta:K\Rightarrow K K$ and $\epsilon:K\Rightarrow \id{\catA}$, called the \textit{comultiplication} and \textit{counit} respectively, such that the duals of the diagrams in Definition \ref{def:monad} commutes.

\end{definition}

Adjunctions are an important source of monads and comonads. For ease of notation we will throughout denote adjunctions by $L\dashv R:\catA\to\catB$, for $L:\catA\to\catB$ left adjoint to $R:\catB\to\catA$.

\begin{remark} \begin{sloppypar}\tolerance 900
    Let $L\dashv R:\catA\to\catB $ be an adjunction with unit $u:\id{\catA}\Rightarrow R  L$ and counit $\epsilon :L  R\Rightarrow \id{\catB}$. By \cite[Lemma 5.1.3]{Riehl2017}, this gives rise to a monad $(R  L,R\whisk \epsilon \whisk L, u)$ on $\catA$, with the whiskered counit $R \whisk \epsilon \whisk L$ serving as the multiplication. Furthermore, the adjunction $L\dashv R$ similarly gives rise to a comonad 
    $(L  R,L\whisk u\whisk R,\epsilon)$ on $\catA$.
    \end{sloppypar}
\end{remark} 

The following examples of  adjunctions and their induced monads will be used in subsequent sections.  

\begin{example}\label{prodmonad}
        For $n$ a positive integer and $\catA$ some category, we denote by $\catA^{\times n}$ the product category whose objects and morphisms are $n$-tuples of objects and morphisms, respectively, in $\catA$. Let $\Delta: \catA\rightarrow \catA^{\times n}$ be the diagonal functor, which takes an object $a$ to the $n$-tuple $(a, \dots, a)$.  Assuming $\catA$ has finite products, we can define the product  functor  $\sqcap:\catA^{\times n}\rightarrow \catA$ as the functor that takes an object $(a_1, \dots, a_n)$ to the product $a_1\times \dots \times a_n\in\ob{\catA}$. The universal property for products guarantees that $\sqcap$ and $\Delta$ are an adjoint pair, with $\sqcap$ being the right adjoint.  Hence, the functor sending $a\in\ob{\catA}$ to $a^{\times n}=a\times \dots \times a$, the  product of $n$ copies of $a$, underlies a monad on $\catA$.
\end{example}

A more general way of obtaining a monad on $\catA$, which will be relevant later in this article, comes from combining an adjunction and monad, as follows. 

\begin{lemma}\label{lem:adjmonad}
For any adjunction $L\dashv R:\catA\to\catB$ and any monad $\monadA=\monadAtriple$ on $\catB$, the composite $R  T  L$ underlies a monad on $\catA$, which we denote $R\monadA L$.
\end{lemma}

\begin{proof}
This is a direct consequence of two classical results about monads and adjunctions. Any monad $\monadA$ arises from an adjunction, via either an Eilenberg-Moore category or a Kleisli category, so $T=R'   L'$ for some adjoint pair $L'\dashv R'.$ The composite $L'  L$ is  left adjoint to the composite $R   R'$, so that $R  R'  L'  L=R  T  L$ underlies a monad on $\catA$. See \cite[4.4.4, 5.2.8, 5.2.11]{Riehl2014} for more details.
\end{proof}

\begin{remark}
 As a consequence of Example \ref{prodmonad} and Lemma \ref{lem:adjmonad}, we see that for any category $\catA$ with finite products and  any positive integer $n$, a monad $\monadA=\monadAtriple$ on $\catA^{\times n}$ gives rise to a monad $\sqcap\monadA\Delta$ on $\catA$ with underlying endofunctor $\sqcap   T  \Delta$.  In particular, the monad of Example \ref{prodmonad} arises in this way from the identity monad on $\catA^{\times n}$ described in Example \ref{ex:identity}.    
\end{remark}

In the final section of the paper, we will need to compare morphisms of monads of the sort constructed in Lemma \ref{lem:adjmonad}.  Let $\monadA_\catA= (T,\mu,\eta)$ and $\monadA_\catB= (T',\mu',\eta')$ be monads on categories $\catA$ and $\catB$, respectively.  
Given adjunctions $L\dashv R:\catA\to \catD$ and $L'\dashv R':\catB\to \catD$, we have a diagram 
\begin{center}
    \begin{tikzcd}[column sep=3em, row sep=3em]

    \catA   
        \arrow[r, "F"]
        \arrow[d, "R"{name=R}, bend left=30]
    & \catB     
        \arrow[d, "R'"{name=Q}, bend left=30]
    \\
    \catD
        \arrow[u, "L"{name=L}, bend left=30]
        \arrow[r, equal]
    &\catD
        \arrow[u, "L'"{name=K}, bend left=30]

        \arrow[phantom, from=L, to=R, "\dashv"]
        \arrow[phantom, from=K, to=Q, "\dashv"].
    \end{tikzcd}
\end{center}

The next lemma provides conditions under which the identity map $\id{\catD}$ is a monad morphism when $F$, together with a natural transformation $\alpha$, is a monad morphism.  

\begin{lemma}\label{lem:conjunctionadjunction} Let  $\monadA_\catA= (T,\mu,\eta)$ and $\monadA_\catB= (T',\mu',\eta')$ be monads on categories $\catA$ and $\catB$, respectively, and let $(F, \alpha):(\catA,\monadA_\catA)\to (\catB, \monadA_\catB)$ be a monad morphism. Let 
     $L\dashv R:\catD\to\catA $ be an adjunction with unit $u:\id{\catD}\Rightarrow R  L$ and counit $\epsilon :L  R\Rightarrow \id{\catA}$, and
   $L'\dashv R':\catD\to\catB $ be an adjunction with unit $u':\id{\catD}\Rightarrow R'  L'$ and counit $\epsilon' :L'  R'\Rightarrow \id{\catB}$.

Let
\[ \tau_L:L' \Rightarrow FL\quad\text{and}\quad \tau_R:R' F\Rightarrow R\]
be natural transformations and 
 write
\begin{align*}
    \tau_R\cdot \tau_L &: R'L' \xRightarrow{R\whisk \tau_L} 
        R'FL
            \xRightarrow{\tau_R\whisk L}
        RL
    \\
    \tau_L\cdot\tau_R &: L'R'F
        \xRightarrow{L'\whisk \tau_R}
        L'R
        \xRightarrow{\tau_L\whisk R}
     FLR
\end{align*}
for the respective compositions of natural transformations.

If $(F\whisk \epsilon)(\tau_L\cdot \tau_R) = \epsilon'\whisk F$ and $(\tau_L\cdot \tau_R) u'=u$, then there is a monad morphism from $R'\monadA_\catB L'$ to $R\monadA_\catA L$, with underlying natural transformation $\beta$ given by the composite
$$R'T'L' \xrightarrow{R' T'\whisk \tau_L} R' T' FL \xrightarrow{R'\whisk \alpha\whisk L} R' F T L \xrightarrow{\tau_R\whisk TL} R TL.$$
\end{lemma}

\begin{proof}
The proof is a direct verification that the pair $(\id{\catD}, \beta)$ satisfy the diagram used to define a monad morphism in Definition \ref{def:monad}.

To see that the natural transformation $\beta$ preserves multiplication, consider the following diagram.
\begin{center}
    \begin{tikzcd}[column sep=6.8em, row sep=3em]
        R'T'L'R'T'L'
            \arrow[r, "R'T'\whisk \epsilon' \whisk T'L'"]
            \arrow[dd, swap, "R'T'L'R'T'\whisk \tau_L"]
            \arrow[rdd, phantom, "(1)" description]
    &
        R'T'T'L'
            \arrow[r, "R'\whisk \mu'\whisk L'"]
            \arrow[dd, "R'T'T'\whisk \tau_L"]
            \arrow[rdd, phantom, "(2)" description]
    &
        R'T'L'
            \arrow[dd, "R'T'\whisk \tau_L"]
\\\\
        R'T'L'R'T'FL
            \arrow[r, "R'T'\whisk \epsilon'\whisk T'FL"]
            \arrow[d, swap, "R'T'L'R'\whisk \alpha \whisk L"]

    &
        R'T'T'FL
            \arrow[r, "R'\whisk \mu'\whisk FL"]
            \arrow[dd, "R'T'\whisk \alpha\whisk L", ""{name=8, near start}]
            \arrow[rdddd, phantom, "(3)" description]
    &
        R'T'FL
            \arrow[dddd, "R'\whisk \alpha\whisk L"]
\\
        R'T'L'R'FTL
            \arrow[rd, "R'T'\whisk\epsilon' \whisk FTL" description, bend left=20]
            \arrow[d, swap,  "R'T'L'\whisk\tau_R\whisk  TL"]
    &
    &
\\
        R'T'L'RTL
            \arrow[d,swap,  "R'T'\whisk \tau_L\whisk RTL"]
            \arrow[r, phantom, "(7)" description]
    &
        R'T'FTL
            \arrow[dd, "R'\whisk\alpha\whisk TL", ""{name=6, near end}]
    &
\\
        R'T'FLRTL
            \arrow[ur, "R'T'F\whisk \epsilon\whisk TL" description, bend right=20]
            \arrow[d, swap, "R'\whisk \alpha\whisk LRTL"]
    &
    &
\\
        R'FTLRTL
            \arrow[r, "R'FT\whisk \epsilon\whisk TL"]
            \arrow[dd, swap, "\tau_R\whisk TLRTL"]
            \arrow[rdd, phantom, "(5)" description]
    &
        R'FTTL
            \arrow[r, "R'F\whisk\mu\whisk L"]
            \arrow[dd, "\tau_R \whisk TTL"]
            \arrow[rdd, phantom, "(4)" description]
    &
        R'FTL
            \arrow[dd, "\tau_R \whisk TL"]
\\\\
        RTLRTL
            \arrow[r, swap, "RT\whisk \epsilon\whisk TL"]
    &
        RTTL
            \arrow[r, swap, "R\whisk \mu\whisk L"]
    &
        RTL.
    \arrow[phantom, near end, "(8)" description, from=3-1, to=8]
    \arrow[phantom, near end, "(6)" description, from=7-1, to=6]
        
    \end{tikzcd}    
\end{center}
Square $(1)$ commutes because both composites equal $R'T'\whisk \epsilon'\whisk T'\whisk \tau_L$.  Similarly, square (2) commutes because both composites equal $R'\whisk \mu' \whisk \tau_L$.  Squares (4), (5), (6), and (8) commute for analogous reasons. 
The assumption that $\alpha$ is the natural transformation part of a morphism of monads implies that square $(3)$ commutes. Finally the assumed compatibility of $\tau_L, \tau_R, \epsilon$ and $\epsilon'$ implies the commutativity of the last part of the diagram, square $(7)$, hence the diagram commutes.  The composite of the left vertical arrows equales the composite of $R'T'L'\whisk \beta$ with $\beta \whisk RTL$, while the right vertical arrows compose to $\beta$.  The top and bottom compositions correspond exactly the the multiplication for $R'T'L'$ and $RTL$ as in Lemma \ref{lem:adjmonad}.  Thus, $(\id{\catD}, \beta)$ preserves multiplication, as desired.

To prove that  $\beta$ commutes with the units of the monads $R' \monadA_\catB L'$ and $R\monadA_\catA L$, we first observe that it fits into the following commuting diagram.
\begin{center}
    \begin{tikzcd}[column sep=2.5em, row sep=2.5em]
        &
       &
            \id{\catD}
                \arrow[dll, swap, bend right=5, "u'"]
                \arrow[drr, bend left=5, "u"]
        &
        &
    \\
        R' L'
            \arrow[rr, "R' \whisk \tau_L"]
           \arrow[dd, swap, "R'\whisk \eta' \whisk  L'"]
       &
       &
       R' F L
           \arrow[rr, "\tau_R\whisk  L"]
           \arrow[ld, "R' \whisk \eta' \whisk  F  L" description]
           \arrow[rd, "R' F  \whisk \eta \whisk L" description]
       &
       &
       R  L  
           \arrow[dd, "R  \whisk \eta\whisk L"]
   \\
       &
       R' T' F L
           \arrow[rr, "R \whisk \alpha \whisk L"]
       &
       &
       R'  F T L
           \arrow[rd, swap, "\tau_R\whisk  T L"]
       &
   \\
       R' T'  L'
           \arrow[ur, swap, "R' T' \whisk \tau_L"]
           \arrow[rrrr, swap, "\beta"]
       &
       &
       &
       &
       R T  L.
   \end{tikzcd}
\end{center}

The left and right triangles commute by naturality of $\eta'$ and $\eta$, respectively, and the middle triangle because $\alpha$ is a component of a monad morphism from $\monadA_\catB$ to $\monadA_\catA$.  The top triangle commutes by hypothesis, and the bottom square commutes by definition of $\beta$.  The commutativity of this diagram shows that the natural transformation $\beta$ commutes with the units of the monads $R' \monadA_\catB L'$ and $R\monadA_\catA L$ as desired, so we can conclude that the natural transformation $\beta$ indeed underlies a morphism of monads.
\end{proof}

\section{$(\mathcal{P}(n),\cap,\n{n})$-modules}
For $n$ a non-negative integer, let  $\n{n}:=\{1,...,n\}$. We denote by $\poset{n}$ the category with subsets of $\n{n}$ as objects and morphisms given by inclusions,  and by convention we set $\poset {0}:=\{\emptyset\}$, the category with one object and no non-identity morphisms. Henceforth, we consider $\poset{n}$ equipped with the symmetric monoidal structure given by intersection and with unit $\n{n}$, i.e., $(\poset{n},\cap, \n{n}\,)$. In this section we study $\poset{n}$-modules with respect to this monoidal structure. We start by recalling the definition of a module over a symmetric monoidal category, specialized to $(\poset{n},\cap, \n{n}\,)$.

\begin{definition}\label{def:nmodule}
    A \emph{$\poset{n}$-module} consists of a category $\catA$ together with a functor $\theta:\catA\times\poset{n}\to\catA$, which we call the \textit{$\poset{n}$-module structure} on $\catA$, that is both unital and associative, in the following sense:
    \begin{enumerate}
         \item (Unital) $\theta(-,\n{n}\,)=\id{\catA}$, and
        \item (Associative) $\theta\left(\theta(a,\subsetA),\subsetB\right)=\theta(a,\subsetA\cap \subsetB)$ for all  $\subsetA,\subsetB\in\ob{\poset{n}}$ and $a\in\ob{\catA}$.
    \end{enumerate}
    
    A \textit{morphism of $\poset{n}$-modules} between two $\poset{n}$-modules $(\catA,\theta_\catA)$ and $(\catB,\theta_\catB)$ consists of a functor $F:\catA\to\catB$ compatible with the $\poset{n}$-module structures, in the sense that the following diagram commutes:
    \begin{center}
        \begin{tikzcd}[column sep=4.5em]
                \catA\times\poset{n}
                    \arrow[r, "F\times\id{\poset{n}}"]
                    \arrow[d, swap, "\theta_\catA"]
            &   \catB\times\poset{n}
                    \arrow[d,  "\theta_\catB"]
            \\
                \catA
                    \arrow[r, swap, "F"]
            &
                \catB.
        \end{tikzcd}
    \end{center}
    We let $\Mod{\poset{n}}$ denote the category of $\poset{n}$-modules and morphisms between them. 
\end{definition}

The following example will be used in Section 6, which concerns dual functor calculus.
\begin{example} \label{ex:theta^n}
    For $\catA$ a category with an initial object $\initial{\catA}$, we can define a functor
    \begin{align*}
        \theta^n:\catA^{\times n}\times \poset{n}&\longrightarrow \catA^{\times n}\\ 
        \left(\left(a_1,\ldots,a_n\right),\subsetA\right)&\longmapsto \left(b_1,\ldots,b_n\right),\\
    \end{align*}
 where
    \[ 
        b_i=\begin{cases}
            a_i & \text{if }i\in \subsetA\\
            \initial{\catA} &\text{if } i\not\in\subsetA.
        \end{cases}
    \]

For any morphism $\big((f_1, \ldots, f_n), \iota \big):\big((a_1, \ldots, a_n), \subsetA\big)\to \big((a_1',\ldots, a_n'), \subsetB\big)$ in $\catA^{\times n}\times \poset{n}$, there is a morphism  $\theta^n\big((f_1, \ldots, f_n), \iota\big)$ in $\catA^{\times n}$ which is given by $f_i$ in the $i$th coordinate if $i\in \subsetA$, and the unique map from the initial object otherwise.

It is not difficult to check that $\theta^n$ is a functor.

    Since $\theta^n\left((a_1,\ldots,a_n),\n{n}\,\right)=(a_1,\ldots,a_n)$ for any $(a_1,\ldots,a_n)\in\ob{\catA^{\times n}}$, and 
    \[ \theta^n\left(\theta^n\left(\left(a_1,\ldots,a_n\right),\subsetA\right),\subsetB\right)=\theta^n\left(\left(a_1,\ldots,a_n\right),\subsetA\cap\subsetB\right),
    \]
    this functor defines a $\poset{n}$-module structure on $\catA^{\times n}$.
\end{example}

To be able to create and understand large classes of $\poset{n}$-modules, it is helpful to characterize them in terms of \textit{coreflective subcategories} and \textit{comonads}. Since we have chosen to consider $\poset{n}$-modules with respect to the symmetric monoidal structure given by intersection instead of union, the following results are dual to those in \cite{HJ21}. 
We start by recalling the remaining notion necessary to formulate these characterizations.

\begin{definition}
    A \emph{coreflective subcategory} of a category $\catA$ is a full subcategory $\catA_0\subseteq \catA$ whose inclusion into $\catA$ admits a right adjoint $R$, which we call the \textit{colocalisation functor}. The unit of such an adjunction is a natural isomorphism $a\cong Ra$ for any $a\in\ob{\catA}$, and if this is in fact the identity, we say that $\catA_0$ is a \textit{strict coreflective subcategory} of $\catA$. 
\end{definition}

\begin{proposition}\label{prop:equivmodulestructure}
The following structures on a small category $\catA$ are equivalent:
\begin{enumerate}
    \item a $\poset{n}$-module structure on $\catA$, 
    \item a set of $n$ strict idempotent comonads $(\comonadA_1,...,\comonadA_n)$ on $\catA$, with $\comonadA_\imath=(K_\imath,\Delta_\imath, \epsilon_\imath)$, such that their underlying endofunctors commute under composition, i.e., $K_\imath  K_k=K_k  K_\imath$ for all $0\leq \imath,k\leq n$, and  
    \item a set of $n$ strict coreflective subcategories of $\catA$, $\{ R_\imath: \catA\to \catA_\imath \}_{\imath=1}^n$, such that $R_\imath|_{\catA_k}  R_k = R_k|_{\catA_\imath}   R_\imath$ for every $1\leq \imath, k\leq n$.
\end{enumerate}
\end{proposition}

\begin{proof}
    Properties $2$ and $3$ are equivalent by \cite[Proposition 1.3]{GabrielZisman}, hence we prove only that properties $1$ and $2$ are equivalent. 
    We first note that by taking the transpose, the existence of a functor $\theta:\catA\times\poset{n}\to \catA$ is equivalent to the existence of an $n$-cube of endofunctors $\tr{\theta}:\poset{n}\to\End{\catA}$. Furthermore, the unit and associativity properties required for $\theta$ to be a $\poset{n}$-module structure on $\catA$, are equivalent to assuming that
    \[\tr{\theta}(\n{n})=\id{\catA} \quad\text{and}\quad \tr{\theta}(\subsetA)  \circ\tr{\theta} (\subsetB)=\tr{\theta}(\subsetA\cap \subsetB) =\tr{\theta}(\subsetB)  \circ \tr{\theta}(\subsetA)\]
    for all $\subsetA,\subsetB\in\ob{\poset{n}}$. In particular, note that this will imply that $\tr{\theta}(\subsetA)  \tr{\theta}\circ(\subsetA)=\tr{\theta}(\subsetA)$ for all $\subsetA\in\ob{\poset{n}}$.

    Before proceeding with the argument, we fix the following notation: $\wo{\imath}:=\n{n}\smallsetminus\{\imath\}$ for $1\leq \imath\leq n$ and $\comp{\subsetA}:=\n{n}\smallsetminus\subsetA$ for any $\subsetA\in\ob{\poset{n}}$. We see that the conditions on the transpose $\tr{\theta}$ above are furthermore equivalent to the following requirements.

    \begin{enumerate}
         \item Each $\tr{\theta}(\subsetA)$ is idempotent, i.e., $\tr{\theta}(\subsetA)\circ  \tr{\theta}(\subsetA)=\tr{\theta}(\subsetA)$, and these functors are equipped with a natural transformation
        \[\epsilon_\subsetA=\tr{\theta}(\subsetA\hookrightarrow \n{n}\,):\tr{\theta}(\subsetA)\Rightarrow \id{\catA}.\]
        \item For any $\subsetA\in\ob{\poset{n}}$ with $\comp{\subsetA}=\{\jmath_1,...,\jmath_k\}$, 
        \[\tr{\theta}(\subsetA)= \tr{\theta}\left(\,\wo{\jmath_1}\,\right) \circ \cdots \circ  \tr{\theta}\left(\,\wo{\jmath_k}\,\right).\]
        \item For all $1\leq \imath,\jmath\leq n$,
        \[\tr{\theta}\left(\,\wo{\imath}\,\right) \circ \tr{\theta}\left(\,\wo{\jmath}\,\right)=\tr{\theta}\left(\,\wo{\jmath}\,\right)\circ  \tr{\theta}\left(\,\wo{\imath}\,\right).\]
    
    \end{enumerate}
    We claim that these conditions are furthermore equivalent to the existence of a collection of $n$ strict idempotent comonads with commuting underlying endofunctors. It is clear that each $\tr{\theta}(\subsetA)$ is such a comonad, hence it is sufficient to show that any such collection $(\comonadA_1,...,\comonadA_n)$ of $n$ comonads on $\catA$, with $\comonadA_\jmath=(K_\jmath,\Delta_\jmath,\epsilon_\jmath)$ gives rise to an $n$-cube of endofunctors $\varphi:\poset{n}\to\End{\catA}$ satisfying the three desired  properties. 
    
    We define this $\varphi$ by
    \begin{align*}
        \varphi:\poset{n}&\longrightarrow \End{\catA}\\
        \subsetA=\n{n}\smallsetminus\left\{\jmath_1,...,\jmath_k\right\}  & \mapsto K_{\jmath_1}\circ \cdots \circ  K_{\jmath_k}.
    \end{align*}
Properties $1$ and $3$ follow by the assumptions that each comonad $\comonadA_\jmath$ is strict idempotent with commuting underlying endofunctors and noting that the composition
\[\epsilon_\subsetA:=\epsilon_{\jmath_1} \cdots  \epsilon_{\jmath_k}=\varphi\left(\subsetA\hookrightarrow \n{n}\,\right):\varphi(\subsetA)\Rightarrow\id{\catA}\]
gives the desired natural transformation. 

Finally, property $2$ follows by observing that since $\varphi\left(\,\wo{\jmath}\,\right)=K_\jmath$, 
\begin{align*}
    \varphi(\subsetA)&=\varphi\left(\,\n{n}\smallsetminus\left\{\jmath_1,...,\jmath_k\right\}\right)\\
    & =K_{\jmath_1}\circ \cdots\circ   K_{\jmath_k}\\
    &= \varphi\left(\,\wo{\jmath_1}\,\right)\circ \cdots\circ   \varphi\left(\,\wo{\jmath_k}\,\right).
\end{align*}
    
\end{proof}
We conclude this section by showing how Example \ref{ex:theta^n} fits into a larger collection of $\poset{n}-$module structures determined by coreflective subcategories of a category $\catA$.
\begin{example}
    Let $\catA$ be a small category and $\mathbf A=\{R_\jmath:\catA\to \catA_\jmath\}^n_{\jmath=1}$ an ordered set of strict coreflective subcategories $\catA_\jmath$ of $\catA$, with $R_\jmath$ the colocalization functors and $\iota_\jmath$ the inclusions.
    
For any $1\leq \jmath\leq n$,  let $\catA^{\times n}_{\wo{\jmath}}$ denote $\catA_1\times\cdots \times \catA_{\jmath -1}\times\catA\times\catA_{\jmath +1}\times\cdots \times \catA_n$, and let
\begin{align*}
    R_{\wo{\jmath}}:=R_1\times\cdots \times R_{\jmath -1}\times \id{\catA}\times R_{\jmath +1}\times \cdots \times R_n:  \catA^{\times n}\to \catA^{\times n}_{\wo{\jmath}}.   
\end{align*}
Then $\left\{R_{\wo{\jmath}}:\catA^{\times n}\to \catA^{\times n}_{\wo{\jmath}}\right\}_{\jmath=1}^n$ is a set of $n$ strict coreflective subcategories of $\catA^{\times n}$, satisfying property 3 of Proposition \ref{prop:equivmodulestructure} and therefore corresponds to a $\poset{n}-$module structure on $\catA^{\times n}$.  This module structure is given by
\begin{align*}
        \theta_{\mathbf A}:\catA^{\times n} \times \poset{n} &\rightarrow \catA^{\times n}\\
        \left(\left(a_1,\ldots,a_n\right),\subsetA\right)&\mapsto \left(b_1,\ldots,b_n\right)\\
    \end{align*}
    where
    \[
    b_\jmath=\begin{cases}
        a_\jmath & \text{if }\jmath\in \subsetA\\
        R_\jmath(a_\jmath) & \text{if } \jmath\not\in \subsetA.
    \end{cases}
    \]
For all $\subsetA\subset \subsetB$, there is a morphism in $\catA^{\times n}$
\[\theta_{\mathbf A}\big((a_1,...,a_n),\subsetA \big) \to \theta_{\mathbf A} \big((a_1,...,a_n),\subsetB\big)\]
which in the $\jmath$th coordinate is given by the counit map $R_\jmath(a_\jmath)\to a_\jmath$ if $\jmath\in\comp{ \subsetA}\cap \subsetB$ and by the identity otherwise.

We can equip the set of strict coreflective subcategories of $\catA$, denoted $\text{CoRefl}(\catA)$, with the following poset structure:
\[
\left(\iota_1\dashv R_1:\catA_1\to\catA \right)<\left( \iota_2\dashv R_2 :\catA_2\to\catA\right)
\]
if and only if there exists a natural transformation $\sigma:\iota_1 R_1\Rightarrow \iota_2 R_2$. 

 Let $\text{CoRefl}_n(\catA)$ denote the set of ordered sets of $n$ strict coreflective subcategories of $\catA$. This set admits a poset structure derived from the on $\text{CoRefl}(\catA)$ as follows:
\[\mathbf A=\left\lbrace \iota_\jmath^{} \dashv R_\jmath:\catA_\jmath\to\catA\right\rbrace_{\jmath =1}^n < \mathbf A'=\left\lbrace\iota_\jmath' \dashv R_\jmath':\catA'_\jmath\to\catA\right\rbrace_{\jmath=1}^n  \]
if and only if
\[\big(\iota_\jmath\dashv R_\jmath:\catA_\jmath\to\catA\big)<\big(\iota_\jmath'\dashv R_\jmath':\catA_\jmath'\to\catA\big)\;\;\text{ for all } 1\leq \jmath \leq n.\] 
By considering $\text{CoRefl}_n(\catA)$ as a category using this poset structure, we can extend the construction of $\theta_{\mathbf A}$ to a functor 
\[
\theta_{(-)}:\text{CoRefl}_n(\catA)\rightarrow\Mod{\poset{n}},
\]
which is given on morphisms as follows: For an inequality $\mathbf A<\mathbf A'$ realized by natural transformations $\left\lbrace\sigma_\jmath:\iota_\jmath R_\jmath\Rightarrow \iota_\jmath'R_\jmath'\right\rbrace_{\jmath=1}^n$,  we have a corresponding natural transformation $\theta_{\mathbf A}\Rightarrow\theta_{\mathbf A'}$ which is given on $((a_1,\ldots,a_n), \subsetA)$ by the identity if $\jmath\in \subsetA$, and $\sigma_\jmath (a_\jmath): R_\jmath (a_\jmath )\rightarrow R_\jmath'(a_\jmath)$ if $\jmath\not\in \subsetA$.

Assume that $\catA$ admits an initial object $\initial{\catA}$, and further denote by $\initial{\catA}$ the subcategory of $\catA$ having this as its unique object. We note that the inclusion $\iota_0:\initial{\catA}\rightarrow\catA$ admits a right adjoint $R_0:\catA\rightarrow\initial{\catA}$. This pair $(\iota_0\dashv R_0:\initial{\catA}\rightarrow\catA)$ is the initial object of the category $\text{CoRefl}(\catA)$, and gives rise to the initial object of $\text{CoRefl}_n(\catA)$ as well

\[
\mathbf A_0:=\left\lbrace \iota_0 \dashv R_0:\initial{\catA}\to\catA\right\rbrace_{\jmath=1}^n.
\]
For any $\mathbf A$ in $\text{CoRefl}_n(\catA)$ we therefore have a unique morphism  $\mathbf A_0\to\mathbf A$, so by applying the functor $\theta_{(-)}$ we obtain a $\poset{n}$-module morphism $\theta_{\mathbf A_0}\rightarrow\theta_{\mathbf A}$. Observe that $\theta_{\mathbf A_0}$ is the $\poset{n}$-module structure described in Example \ref{ex:theta^n}.

\end{example}

\section{Systems of homotopy colimits}

Our method for constructing monads from $\poset{n}$-modules uses right Kan extensions and homotopy colimits. In fact, we need models for homotopy colimits that satisfy certain standard properties up to natural isomorphism rather than weak equivalence.  We call such models  \emph{systems of homotopy colimits}, and use this section to define them and provide examples.  We also  review right Kan extensions, and establish two important ways in which right Kan extensions interact with systems of homotopy colimits.  

Throughout this section,  $(\modcatA,\we)$ denotes a category $\modcatA$ equipped with a distinguished wide subcategory $\we$, the morphisms of which we refer to as \textit{weak equivalences}, and we assume that $\modcatA$ has a terminal object, which we denote by $\terminal{\modcatA}$. In the following we work in the category $\Cat$, which consists of small categories and functors between them. 

\begin{definition}\label{def:hocolim}
    A \emph{system of homotopy colimits} on $(\modcatA, \we)$ is a collection of functors
    \[
    \syshocolim:=\{\hocolim_{\indexI}:\fun{\indexI}{\modcatA}\rightarrow\modcatA\,|\,\indexI\in \ob{\Cat}\}
    \]
    such that the following properties hold.
    \begin{enumerate}
        \item (Fubini) For all $\indexI,\indexJ\in\ob{\Cat}$ and functors $F:\indexI\times\indexJ\rightarrow \modcatA$, there exist natural isomorphisms
        \[
        \hocolim_{\indexJ}\ \hocolim_{\indexI}\ F\cong\hocolim_{\indexI\times\indexJ}\ F\cong\hocolim_{\indexI}\ \hocolim_{\indexJ}\ F.
        \]
        \item\label{def:hocolim2} For any functor $\alpha:\indexI\to\indexJ$, there exists a natural transformation $\Tilde{\alpha}$

\begin{center}
\begin{tikzcd}[row sep=1em, column sep=1em]
\fun{\indexJ}{\modcatA} 
    \ar[ddr, swap, "\alpha^*", ""{name=alpha}]
    \ar[rr,  "\hocolim_{\indexJ}", ""{name=F}]
&&
\modcatA,
\\\\
&
\fun{\indexI}{\modcatA}
    \ar[ruu, "\hocolim_{\indexI}", swap, ""{name=G}]
    \ar[uu, Rightarrow, shorten=3mm, to=F,  "\Tilde{\alpha}\;"]
\end{tikzcd}

\end{center}
where $\alpha^*$ denotes the functor given by precomposition with $\alpha$.

        \item If $\mbox{Const}_{\terminal{\modcatA}}:\indexI\rightarrow\modcatA$ is the constant functor at $\terminal{\modcatA}$ for any $\indexI\in\ob{\Cat}$, then there exists a natural isomorphism
        \[
        \hocolimunder{\indexI}\ (\mbox{Const}_{\terminal{\modcatA}})\cong\terminal{\modcatA}.
        \]

        \item\label{def:hocolim4}If $\poset{0}$ denotes the category with one object, denoted $\emptyset$, and no non-identity morphisms, then there exists a natural isomorphism
        \[
        \hocolimunder{\poset{0}}\ F\cong F(\emptyset)
        \]
        for all functors $F:\poset{0}\ \rightarrow\modcatA$.

        \item Let $F,G\in \ob{\fun{\indexI}{\modcatA}}$. For any natural transformation $F\Rightarrow G$ that is an objectwise weak equivalence, 
       the induced map $\hocolim_{\indexI}F\rightarrow\hocolim_{\indexI}G$ is a weak equivalence as well. 
    \end{enumerate}
\end{definition}

\begin{notation}
    A  category $\modcatA$ that has a terminal object and is equipped with a distinguished wide subcategory $\we$ and a system of homotopy colimits $\syshocolim$ is denoted by a triple $(\modcatA,\we,\syshocolim)$.
\end{notation}

Any cocomplete category that has a terminal object also has an associated system of homotopy colimits.
\begin{example}
Suppose $\modcatA$ is  cocomplete and has a terminal object.  Then standard properties of colimits ensure that 
$$
\{{\operatorname{colim}}_\indexI:\fun{\indexI}{\modcatA}\rightarrow\modcatA\,|\,\indexI\in \ob{\Cat}\}
$$
is a system of homotopy colimits for $\modcatA$, when the weak equivalences are the isomorphisms in $\modcatA$.  
\end{example}

As the next two examples show, standard models for homotopy colimits can be used to form systems of homotopy colimits.

\begin{example} 
Suppose $\modcatA$ is a simplicial model category in which every object is cofibrant.  For a small category $\indexI$ and an $\indexI$-diagram $X:\indexI\to\modcatA$, one can use \cite[Definition 18.1.2]{H03}  to define $\hocolimunder{\indexI}X$. It follows by arguments dual to those used in Appendix A of \cite{BJM15}, that properties 1 through 4 are satisfied by this model for homotopy colimits.  Property 5 is \cite[Theorem 18.5.4(1)]{H03}. If $\modcatA$ is a simplicial model category with functorial cofibrant replacement, the condition that every object is cofibrant can be removed, provided that each diagram $X$ is postcomposed with the cofibrant replacement functor before applying  \cite[Definition 18.1.2]{H03}. 
\end{example} 

\begin{example} Another standard model for homotopy colimits of diagrams in simplicial model categories is the geometric realization of the simplicial bar construction of the diagram (see \cite[Chapter 5]{Riehl2014}).  Arakawa showed in \cite{arakawa} that this model can be adapted for use in $Ch(\catA)$, equipped with the injective model structure. Here $\catA$ is a Grothendieck abelian category, and $Ch(\catA)$ is the category of chain complexes in $\catA$. The injective model structure on $Ch(\catA)$ has as its weak equivalences the quasi-isomorphisms of chain complexes.   For an $\indexI$-diagram  in $Ch(\catA)$, its bar construction is a simplicial object in $Ch(\catA)$ and by normalizing this in the simplicial direction one obtains a bicomplex in $\catA$ whose total complex serves as a model for the homotopy colimit of the diagram \cite[Corollary 3.14]{arakawa}.  One can check that this model satisfies the properties of Definition \ref{def:hocolim} and provides $Ch(\catA)$ with a system of homotopy colimits. If one uses the Moore complex instead of the normalized chain complex construction, one also gets a model for homotopy colimits in $Ch(\catA)$, which however does not satisfy axioms 3 and 4 in Definition \ref{def:hocolim}.
\end{example}

In the next sections of the paper, systems of homotopy colimits will be used in conjunction with right Kan extensions.  The remainder of this section will be used to review right Kan extensions and prove Lemma \ref{lem:hocolim}, in which we establish useful properties of the relationship between  homotopy colimits and  right Kan extensions.

\begin{definition}\cite[Definition 6.1.1]{Riehl2017}\label{def:Kan}
 The \textit{right Kan extension} of $F:\indexI\to\catB$ along $\gamma:\indexI\to\catA$ is the initial pair $(\Ran{F}{\gamma},\epsilon)$ consisting of a functor $\Ran{F}{\gamma}:\catA \rightarrow \catB$ together with a natural transformation $\epsilon: (\Ran{F}{\gamma})  K \Rightarrow F$, in the sense that any natural transformation $\delta:G  \gamma\Rightarrow F$, where $G:\catA\to\catB$ is a functor, factors uniquely through $\epsilon$, as illustrated in the following diagrams:

\begin{center}
    \begin{tikzcd}
\indexI  
    \ar[ddr, "\gamma"', ""{name=K}]
    \ar[rr, "F", ""{name=F}]
&&
\catB
\\\\
&
\catA
    \ar[ruu, "G", swap, ""{name=G}]
    \ar[uu, Rightarrow, shorten=5mm, to=F,  "\delta \;"]
\end{tikzcd}
=
 \begin{tikzcd}[column sep=3em]
\indexI  
    \ar[ddr, "\gamma"', ""{name=K}]
    \ar[rr, "F", ""{name=F}]
&&
\catB.
\\\\
&
\catA
    \ar[ruu, bend left=15, "\Ran{F}{\gamma}" description, ""{name=Ran, below}]
    \ar[ruu, bend right=50, swap, "G", ""{name=G} ]
    \ar[uu, Rightarrow, shift left, shorten=5mm, to=F,  "\epsilon\;"]
    \arrow[Rightarrow, "\exists!", shorten=2mm, from=G, to=Ran]
\end{tikzcd}
\end{center}

\end{definition}

Right Kan extensions need not exist, in general.  The following theorem provides a sufficient condition for existence and an explicit formulation of the right Kan extension of $F$ along $\gamma$.  In the  theorem, $a\downarrow \gamma$  denotes the comma category whose objects are pairs consisting of an object $i\in \ob{\indexI}$ and a morphism $a\to \gamma(i)$ in $\catA$; see \cite{Riehl2017} for details. 
We denote by $\text{pr}_a: a\downarrow \gamma \to \indexI$  the evident projection functor.

\begin{theorem}{\cite[Theorem 6.2.1]{Riehl2017}}\label{theorem:Kanlimits}
Given functors $\gamma:\indexI\to\catA$ and $F:\indexI\to\catB$ such that for every $a\in\ob{\catA}$ the limit
\[
\Ran{F}{\gamma}(a)=\lim\left(a\downarrow \gamma\xrightarrow{\text{pr}_a}\indexI\xrightarrow{F}\catB\right)
\]
exists in $\catA$, these limits define the right Kan extension $\Ran{F}{\gamma}:\catA\to\catB$, and the natural transformation $\epsilon:\Ran{F}{\gamma}  \gamma\Rightarrow F$ can be extracted from the limit cone.
\end{theorem}

  In fact, the right Kan extension is functorial in the sense that there exists a functor
\[ \Ran{}{\gamma}: \fun{\indexI}{\catB} \rightarrow \fun{\catA}{\catB}.\]
See \cite[Proposition 6.5]{Riehl2017} for further details. 

\begin{remark}\label{Remark:Kanexsistence}
    It follows directly by Theorem \ref{theorem:Kanlimits} that if $\catA$ is a small category, and $\catB$ is complete, then the right Kan extension of $F:\indexI\to\catB$ along $\gamma:\indexI\to\catA$ always exists. To ease the assumptions we will therefore throughout usually assume $\catA$ is small and $\catB$ is complete to ensure existence of right Kan extension, but these assumptions can often be relaxed to existence of the limits in Theorem \ref{theorem:Kanlimits}.
\end{remark}

\begin{lemma}\label{lem:hocolim} Let $(\modcatA,\we,\syshocolim)$ be a complete category with terminal object $\terminal{M}$ and a system of homotopy colimits, and let $\gamma:\indexI\rightarrow\catA$ and $\beta:\indexJ\rightarrow\catB$ be two functors satisfying the hypotheses of Theorem \ref{theorem:Kanlimits}.
        \begin{enumerate}
            \item\label{lem:hocolimprop1}  If $\gamma$ is injective on objects and fully faithful, and $\catA(a,\gamma(i))=\emptyset$ for all $a\not\in \im(\gamma)$, $i\in\ob{\indexI}$, then for any functor $F:\indexI\times\indexJ\rightarrow\modcatA$, there exists a natural isomorphism
            \[
            \hocolimunder{\catA\times\catB}\ \Ran{}{\gamma\times\beta}F\cong\hocolimunder{\catA}\ \Ran{}{\gamma}\left(\hocolimunder{\catB}\ \Ran{}{Id_\indexI\times\beta}F\right).
            \]

            \item \label{lem:hocolimprop2} Given $\pi:\indexI\rightarrow\indexJ$ and $\tau:\catA\rightarrow\catB$ such that the following diagram commutes 
            \begin{center}
                \begin{tikzcd}
                    \indexI \arrow[r, "\pi"] \arrow[d, swap, "\gamma"] & \indexJ \arrow[d, "\beta"]\\
                    \catA \arrow[r,swap,"\tau"]& \catB, \\
                \end{tikzcd}
            \end{center}
            together with a natural transformation $\sigma:\Ran{}{\gamma} \pi^\ast\Rightarrow\tau^\ast \Ran{}{\beta}$, then for any functor $F:\indexJ\to\modcatA$, there exists an induced morphism
            \[
            \hocolimunder\catA\ \Ran{}{\gamma}\pi^\ast F\rightarrow\hocolimunder{\catB}\ \Ran{}{\beta}F.
            \]
        \end{enumerate}

\end{lemma}
\begin{proof}
\textit{1.} Since $\gamma$ is injective on objects and fully faithful, and $\catA(a,\gamma(i))=\emptyset$ for all $a\not\in \text{Im}(\gamma)$, the formula for right Kan extension in Theorem \ref{theorem:Kanlimits} implies that for any $G:\indexI\rightarrow\catC$, with $\catC$ a complete category, and any $a\in\ob{\catA}$,  
\[
\Ran{G}{\gamma}(a)=\begin{cases}
    G(i) & \text{if } a=\gamma(i)\\
    \terminal{\modcatA} & \text{if }a\not\in \im(\gamma).
\end{cases}
\]

In particular, 
\[
    \Ran{}{\gamma}\hocolimunder{\catB}\ \Ran{}{Id_\indexI\times\beta}F(a)=\begin{cases}
    \hocolimunder{\catB}\ \Ran{}{Id_\indexI\times\beta}F(i) & \text{if }a=\gamma(i)\\
    \terminal{\modcatA} & \text{if }a\not\in \im(\gamma).
\end{cases}
\]

By property 1 of Definition \ref{def:hocolim}, 
\[
\hocolimunder{\catA\times\catB}\ \Ran{}{\gamma\times\beta}F\cong\hocolimunder{\catA}\ \hocolimunder{\catB}\ \Ran{}{\gamma\times\beta}F.
\]

Again by the formula for right Kan extensions, we can for any $a\in\ob{\catA}$, $b\in\ob{\catB}$, write
\begin{align*}
    \Ran{}{\gamma\times\beta}F(a,b)&\cong\operatorname{lim}_{(a,b)\downarrow \gamma\times\beta}\left(F \text{pr}_{(a,b)}\right)\\
&\cong\operatorname{lim}_{a\downarrow\gamma}\left(\operatorname{lim}_{b\downarrow\beta}(F \text{pr}_b)\right) \text{pr}_a\\
    &\cong\begin{cases}
        \Ran{}{Id_\indexI\times\beta}F(i,b) &\text{if } a=\gamma(i)\\
        \terminal{\modcatA} &\text{if } a\not\in \im(\gamma),\\
    \end{cases}
\end{align*}

so by property 3 of Definition \ref{def:hocolim}, 
\begin{align*}
    \hocolimunder{\catB}\ \Ran{}{\gamma\times\beta}F(a)&\cong\begin{cases}
    \hocolimunder{\catB}\ \Ran{}{Id_\indexI\times\beta}F(i) & \text{if }a=\gamma(i)\\
    \terminal{\modcatA} & \text{if }a\not\in \im(\gamma)
\end{cases}\\
&\cong\Ran{}{\gamma}\hocolimunder{\catB}\ \Ran{}{Id_\indexI\times\beta}F(a)\\
\end{align*}
and we can conclude the proof of part 1 of the lemma.

\textit{2.} The component of the natural transformation $\sigma$ at $F\in\ob{\fun{\indexJ}{\modcatA}}$ is a natural transformation $\sigma_F:\Ran{(F\pi)}{\gamma}\Rightarrow (\Ran{F}{\beta})  \tau$. By the functoriality of $\hocolimunder{\catA}$, there is an induced morphism
\[\hocolimunder{\catA}\ \Ran{(F  \pi)}{\gamma} \rightarrow \hocolimunder{\catA}\ (\Ran{F}{\beta})  \tau. \]
By property 2 of Definition \ref{def:hocolim}, there exists a morphism
\[ \widetilde{\tau}:\hocolimunder{\catA}\ (\Ran{F}{\beta})  \tau \rightarrow \hocolimunder{\catB}\ \Ran{F}{\beta}.\] 
The desired morphism is the composite of these two.
\end{proof}

\section{Main theorem}
In this section, we complete the process of deriving monads from cubical diagrams.  Starting with a $\poset{n}$-module on a category $\catA$, we use the iterated homotopy cofibers of cubical diagrams arising from this $\poset{n}$-module  to construct a monad on $\fun{\catA}{\modcatA}$ for any complete category  with a system of homotopy colimits $(\modcatA,\we,\syshocolim)$.  We achieve this goal by using a model for the iterated homotopy cofiber  that allows us to realize it as a single homotopy colimit in a functorial fashion. We develop  this model in Section \ref{section:hocofibermodel}. In Section \ref{section:mainthm}, we construct the monads and prove our main result.  Throughout this section, we assume that that $(\modcatA,\we,\syshocolim)$ is a complete category with a system of homotopy colimits.

\subsection{A model for iterated homotopy cofibers}\label{section:hocofibermodel}

In this section, we describe our model for the iterated homotopy cofiber of a cubical diagram.  To construct this model, we  generate an associated expanded cubical diagram, and then take the homotopy colimit of this expanded diagram. We begin by providing an alternate description of $n$-cubical diagrams, which we use  to obtain  the desired expanded diagram for an $n$-cubical diagram $\chi:\poset{n}\to\modcatA$.

Let $\I$ be the category with two objects, denoted by $0$ and $1$, and a single non-identity morphism $t:0\to1$. If we equip this with the symmetric monoidal product $\dicei$, for which  $1$ is the unit, and $0\dicei0=0$, then the morphism $\omega_1:(\poset{1},\cap)\to (\I,\dicei)$, which sends $\emptyset$ to $0$ and $\{1\}$ to 1 is a monoidal isomorphism. By equipping $\I^{\times n}$ with the component-wise symmetric monoidal product extending $\dicei$, this isomorphism can be extended to a monoidal isomorphism $\omega_n:(\poset{n},\cap)\to (\I^{\times n},\dicei)$ as follows.

For any $\subsetA\in\ob{\poset{n}}$, let $\textrm{char}_\subsetA$ be the characteristic function of $\subsetA$, i.e., for $1\leq i\leq n$, $\textrm{char}_\subsetA:\n{n}\to\{0,1\}$ is given by
\[
\textrm{char}_\subsetA(i)=\begin{cases}
    0 & \textrm{ if } i\notin \subsetA\\
    1 & \textrm{ if } i\in \subsetA.
\end{cases}
\]
We define $\omega_n:\poset{n}\to \I^{\times n}$ on objects $\subsetA\in\ob{\poset{n}}$ by
\[
\omega_n(\subsetA)=(\textrm{char}_\subsetA(1),\dots,\textrm{char}_\subsetA(n)),
\]
and extend to morphisms in the obvious way.
It is clear that $\omega_n$ is an isomorphism. To see that $\omega_n$ is monoidal, consider two subsets $\subsetA,\subsetB\in\ob{\poset{n}}$. Using our description of $\dicei$, we see that

\[
\begin{split}
    \omega_n(\subsetA\cap \subsetB)&=(\textrm{char}_{\subsetA\cap \subsetB}(1),\dots,\textrm{char}_{\subsetA\cap \subsetB}(n))\\
    &=(\textrm{char}_{\subsetA}(1)\dicei\textrm{char}_{\subsetB}(1),\dots,\textrm{char}_{\subsetA}(n)\dicei\textrm{char}_{\subsetB}(n)))\\
    &=\omega_n(\subsetA)\dicei\omega_n(\subsetB).
\end{split}
\]

The category $\I$ embeds into another symmetric monoidal category, denoted by $(\Lambda,\dicei,1)$, which is the category with objects $\{0,1,\mathscr{l}\}$, and non-identity morphisms given by
\[\mathscr{l}\xleftarrow{s}0\xrightarrow{t}1.\]

The symmetric monoidal product on $\I$ extends to $\Lambda$ by setting $\mathscr{l}\dicei d=\mathscr{l}$ for all $d\in\ob{\Lambda}$, and we further extend this to $\Lambda^{\times n}$ component-wise. 
For example, the category $\Lambda\times\Lambda$ is given by the diagram
\[
\begin{tikzcd}
    (1,1)&(0,1)\arrow{l}\arrow{r}&(\mathscr{l}, 1)\\
    (1,0)\arrow{d}\arrow{u}&(0,0)\arrow{l}\arrow{r}\arrow{d}\arrow{u}&(\mathscr{l},0)\arrow{d}\arrow{u}\\
    (1,\mathscr{l})&(0,\mathscr{l})\arrow{l}\arrow{r}&(\mathscr{l},\mathscr{l}).\\
\end{tikzcd}
\]

Let $\iota:\I\to\Lambda$ denote the inclusion functor, and set 
\[
\phi_n=\iota^{\times n} \omega_n:\poset{n}\to\Lambda^{\times n}.
\]
It is along this composite functor that we will extend $n$-cubical diagrams.

\begin{remark}

    Let $\chi:\poset{n}\to\modcatA$ be an $n$-cubical diagram, with $\modcatA$ a complete category.  The extended $\Lambda^{\times n}$-diagram of $\chi$ is then the right Kan extension of $\chi$ along $\phi_n:\poset{n}\to \Lambda^{\times n}$:

\[\begin{tikzcd}
    \poset{n}\arrow{rr}{\chi}\arrow{dr}[swap]{\phi_n}&& \modcatA.\\
    &\Lambda^{\times n}\arrow{ur}[swap]{\Ran{\chi}{\phi_n}}&\\
\end{tikzcd}\]

Since $\modcatA$ is complete by hypothesis, and $\Lambda^{\times n}$ is small, it follows by Remark \ref{Remark:Kanexsistence} that this right Kan extension does indeed exists.

We can apply Theorem \ref{theorem:Kanlimits}, which lets us derive, for any object $(l_1,\dots,l_n)\in\ob{\Lambda^{\times n}}$, a formula for $\Ran{\chi}{\phi_n}(l_1,\dots,l_n)$. If $\underline{l}=(l_1,\dots,l_n)\in \ob{\I^{\times n}}$, then $\underline{l}\downarrow\phi_n = \subsetA$, where $\subsetA\in\ob{\poset{n}}$ contains $i$ if $l_i=1$, and does not contain $i$ if $l_i=0$.  So,

\[
\begin{split}
   \Ran{\chi}{\phi_n}(l_1,\dots,l_n)&=\lim\left(\subsetA\xrightarrow{\text{pr}_{\underline{l}}}\poset{n}\xrightarrow{\chi}\modcatA\right) \\
   &= \chi(\subsetA).
\end{split}
\]

If $l_i=\mathscr{l}$ for some $i$, then $\underline{l}\downarrow\phi_n$ is the empty category, and 
\[
\begin{split}
   \Ran{\chi}{\phi_n}(l_1,\dots,l_n)&=\lim\left(\emptyset\xrightarrow{\text{pr}_{\underline{l}}}\poset{n}\xrightarrow{\chi}\modcatA\right) \\
   &=\operatorname{lim}_\emptyset\left(\poset{n}\xrightarrow{\chi}\modcatA\right) \\
   &=\terminal{\modcatA}.
\end{split}
\]
By noting that there exists a unique $\subsetA\in\ob{\poset{n}}$ satisfying $\phi_n(\subsetA)=(l_1,\dots,l_n)$ exactly when $(l_1,\dots, l_n)\in\ob{\I^{\times n}}$, we obtain the following formula.
\begin{equation}\label{eq:ranformula}
\Ran{\chi}{\phi_n}(l_1,\dots,l_n)=\begin{cases}
    \chi(\subsetA)&\textrm{ if } (l_1,\dots,l_n)=\phi_n(\subsetA)\\
    \terminal{\modcatA}& \textrm{ if } (l_1,\dots,l_n)\notin\text{Im}(\phi_n).
\end{cases}
\end{equation}
In other words, $\Ran{\chi}{\phi_n}(l_1,\dots,l_n)$ is the terminal object unless there exists an $\subsetA\in\ob{\poset{n}}$ such that $\phi_n(\subsetA)=(l_1,\dots,l_n)$, which justifies considering this right Kan extension as an extended $n$-cubical diagram of $\chi$.

\end{remark}

\begin{definition}
    Let $(\modcatA,\we,\syshocolim)$ be a complete category equipped with a system of homotopy colimits.  The \emph{iterated homotopy cofiber} of an $n$-cubical diagram $\chi:\poset{n}\to\modcatA$, denoted by $\icofib \chi$, is given by
    \[
    \icofib \chi=\hocolim_{\Lambda^{\times n}}\Ran{\chi}{\phi_n}.
    \]
\end{definition}

\subsection{Obtaining monads from modules}\label{section:mainthm}

In this section, we state and prove Theorem \Cref{theorem:pnmodstomonads}, which gives a functor between the categories of $\poset{n}$-modules and monads.  We begin by defining the endofunctor underlying the monad associated to a $\poset{n}$-module. 

Let $(\modcatA,\we,\syshocolim)$ be a complete category equipped with a system of homotopy colimits as in Definition \ref{def:hocolim} and let $(\catA,\theta)$ be a $\poset{n}$-module.  The functor $\theta$ is equivalently given by an $n$-cube of endofunctors 
\[ \tr{\theta}: \poset{n}\to \End{\catA}\] as in the proof of Proposition \ref{prop:equivmodulestructure}.  Given a functor $F:\catA\to \modcatA$, we obtain an $n$-cube of functors
\begin{equation}\label{e:ncubefrommodule}
\begin{split}
    \theintcube{(-)}{\theta}{F}:\poset{n}&\to\fun{\catA}{\modcatA}\\
    \subsetA&\mapsto F \tr{\theta}(\subsetA)
\end{split}
\end{equation}
by post-composition with $F$. Noting that $\fun{\catA}{\modcatA}$ is again complete, we may apply the iterated homotopy cofiber to obtain a functor $\icofib \theintcube{(-)}{\theta}{F}:\catA \to \modcatA$, which further extends to an endofunctor on the category $\fun{\catA}{\modcatA}$:
\[
\begin{split}
    \ifibcube{\theta}{\modcatA}:\fun{\catA}{\modcatA}&\to\fun{\catA}{\modcatA}\\
    F&\mapsto\icofib \theintcube{(-)}{\theta}{F}.
\end{split}
\]

It is this endofunctor obtained from a $\poset{n}$-module that we wish to show underlies a monad $\left(  \ifibcube{\theta}{\modcatA},\mult,\eta\right)$ on $\fun{\catA}{\modcatA}$.  We show, in the proof of the next lemma, that the unit and multiplication for the monad essentially arise from the symmetric monoidal structures of $\poset{n}$ and $\Lambda^{\times n}.$

\begin{proposition}\label{lem:Tmonad}
  Let $(\modcatA,\we,\syshocolim)$ be a complete category equipped with a system of homotopy colimits, and let $(\catA,\theta)$ be a $\poset{n}$-module. Then $\ifibcube{\theta}{\modcatA}$ underlies a monad $\left(\ifibcube{\theta}{\modcatA},\mult,\eta\right)$ on $\fun{\catA}{\modcatA}$.
\end{proposition}

\begin{proof}
    
For ease of notation, let $T:=\ifibcube{\theta}{\modcatA}$. 

We first wish to define the unit $\eta:\id{\fun{\catA}{\modcatA}}\Rightarrow T$. To do so, we define the functor
\[
\begin{split}
  \eta^n: \poset{0}& \to \Lambda^{\times n}   \\
  \emptyset&\mapsto (1,\dots,1).
\end{split}
\]
  Using property \ref{def:hocolim2} of Definition \ref{def:hocolim}, there exists a natural transformation $\widetilde{\eta^n}$ 

  \[\begin{tikzcd}[row sep=large, column sep=1em] 
\fun{\Lambda^{\times n}}{\modcatA} \arrow[dr, "(\eta^n)^*"'{name=K}] 
 \arrow[rr, "\hocolim_{\Lambda^{\times n}}", ""{name=F, below}] && \modcatA. \\ 
 & |[alias=C]| \fun{\poset{0}}{\modcatA} \arrow[ur, swap,
 "\hocolim_{\poset{0}}"] 
 \arrow[Rightarrow, swap, from=C, to=F, "\;\widetilde{\eta^n}",shorten >=1em,shorten <=1em] 
\end{tikzcd}\]

For $F\in\ob{\fun{\catA}{\modcatA}}$ and $a\in\ob{\catA}$, we define $\eta_F(a)$ as the evaluation of this natural transformation $\widetilde{\eta^n}$ at $\Ran{F(\theta(a,-))}{\phi_n}$:
\[
\eta_F(a):\hocolim_{\poset{0}}\Ran{F(\theta(a,-))}{\phi_n} \eta^n\Rightarrow\hocolim_{\Lambda^{\times n}}\Ran{F(\theta(a,-))}{\phi_n}=TF(a),
\]
noting that this is natural in both $F$ and $a$. The description of the right Kan extension expressed in Equation (\ref{eq:ranformula}) implies that
\[\Ran{F(\theta(a,-))}{\varphi_n}  \eta^n = \Ran{F(\theta(a,\n{n}))}{\varphi_n}=F(a),\]
and by property 4 of Definition \ref{def:hocolim} it further follows that $\hocolimunder{\poset{0}}\ \Ran{F(a)}{\varphi_n}\cong F(a)$, hence this indeed defines a natural transformation $\eta:\id{\fun{\catA}{\modcatA}}\Rightarrow T$. 

To define the multiplication $\mu:T  T\Rightarrow T$, we first wish to show that for any $F\in\ob{\fun{\catA}{\modcatA}}$, $a\in\ob{\catA}$,  
$$T(TF)(a)\cong\hocolimunder{\Lambda^{\times n}\times \Lambda^{\times n}}\Ran{\cap^\ast F(\theta(a,-))}{\varphi_n\times\varphi_n}.$$ To do so, we note that the definition of $T$ and the assumption that $\theta$ is a $\poset{n}$-module structure on $\catA$ together imply that
\[
\begin{split}
    T(TF)(a)&=\hocolim_{\Lambda^{\times n}}\Ran{TF(\theta(a,-))}{\phi_n}\\
    &=\hocolim_{\Lambda^{\times n}}\Ran{\hocolim_{\Lambda^{\times n}}\Ran{F(\theta(\theta(a,-),-))}{\id{\mathcal{P}(\underline{n})}\times\phi_n}}{\phi_n}\\
    &=\hocolim_{\Lambda^{\times n}}\Ran{\hocolim_{\Lambda^{\times n}}\Ran{F(\theta(a,-\cap -))}{\id{\mathcal{P}(\underline{n})}\times\phi_n}}{\phi_n}.
    \end{split}
\]
Next we wish to apply property 1 of Lemma \ref{lem:hocolim}. It is clear that $\varphi_n$ is injective on objects and fully faithful, since per definition $\varphi_n=\iota^{\times n}  \omega_n$ with $\iota$ the inclusion and $\omega_n$ a monoidal isomorphism.
To see that $\Lambda^{\times n}((l_1,...,l_n),\varphi_n(\subsetA))=\emptyset$ for all $(l_1,...,l_n)\notin \text{Im}(\varphi_n)$ and $\subsetA\in\ob{\poset{n}}$, we note that $\text{Im}(\varphi_n)=\{(t_1,...,t_n)|t_i=0\text{ or } 1\}$, so $l_i=\mathscr{l}$ for at least one $i$. Since $\Lambda(\mathscr{l},0)=\Lambda(\mathscr{l},1)=\emptyset$, the desired equality follows. We can therefore apply part 1 of Lemma \cref{lem:hocolim} from which it follows that     
\[ 
\begin{split}
\hocolim_{\Lambda^{\times n}}&\Ran{}{\phi_n}{\hocolim_{\Lambda^{\times n}}\Ran{F(\theta(a,-\cap-))}{\id{\mathcal{P}(\underline{n})}\times\phi_n}}
    \\
    \cong&\hocolim_{\Lambda^{\times n}\times\Lambda^{\times n}}\Ran{F(\theta(a,-)) \cap}{\phi_n\times \phi_n}\\
     =&\hocolim_{\Lambda^{\times n}\times\Lambda^{\times n}}\Ran{\cap^\ast F(\theta(a,-))}{\phi_n\times \phi_n}.
    \end{split}
\]
It follows that $T(TF)(a)\cong\hocolimunder{\Lambda^{\times n}\times \Lambda^{\times n}}\Ran{\cap^\ast F(\theta(a,-))}{\varphi_n\times\varphi_n}$. 

Next we wish to apply part 2 of Lemma \ref{lem:hocolim}, for which we need to prove the existence of a natural transformation
\begin{equation}\label{eq:alpha}
    \Ran{\cap^\ast(-)}{\varphi_n\times\varphi_n}\Rightarrow \dicei^\ast\Ran{(-)}{\varphi_n},
\end{equation}
between functors $\fun{\poset{n}}{\modcatA}\to\fun{\Lambda^{\times n}\times \Lambda^{\times n}}{\modcatA}$. Let $G\in\ob{\fun{\poset{n}}{\modcatA}}$, and note that the following diagram commutes:
\begin{center}
    \begin{tikzcd}
        \poset{n}\times \poset{n}
            \arrow[r, "\cap"]
            \arrow[d, swap, "\varphi_n\times\varphi_n"]
        &
        \poset{n}
            \arrow[d, "\varphi_n"]
        \\
        \Lambda^{\times n}\times \Lambda^{\times n}
            \arrow[r, swap, "\dicei"]
        &
        \Lambda^{\times n}.
    \end{tikzcd}
\end{center}
The existence of the desired natural transformation follows by applying the universal property of right Kan extensions to the following diagram
\begin{center}
    \begin{tikzcd}
        \poset{n}\times \poset{n}
            \arrow[r, "\cap"]
            \arrow[rdd, swap, "\varphi_n\times\varphi_n"]
        &
        \poset{n}
            \arrow[r, "G"]
            \arrow[d, "\varphi_n"]
        &
        \modcatA.\\
        &
        \Lambda^{\times n}
            \arrow[ur, "\Ran{G}{\varphi_n}" description]
        &
        \\
        &
        \Lambda^{\times n}\times \Lambda^{\times n}
            \arrow[u, swap, "\dicei"]
            \arrow[uur,  swap, bend right, "\Ran{\cap^\ast G}{\varphi_n\times \varphi_n}"]
        &        
    \end{tikzcd}
\end{center}
Part \ref{lem:hocolimprop2} of Lemma \ref{lem:hocolim} then implies that we have a morphism 
\[\mu_F(a): T(TF)(a)\to TF(a),\]
which is natural in $F$ and given by
\[ \hocolim_{\Lambda^{\times n}\times\Lambda^{\times n}}\Ran{\cap^\ast F(\theta(a,-))}{\phi_n\times \phi_n} \xrightarrow{\widetilde{\dicei}}\hocolim_{\Lambda^{\times n}}\Ran{F(\theta(a,-))}{\phi_n}.\]
It follows from the associativity and unitality of the monoidal structure $\dicei$ on $\Lambda^{\times n}$ that $\mu_F$ is both unital with respect to $\eta_F$ and associative.

\end{proof}

The association of a $\poset{n}$-module $(\catA, \theta)$ to the monad $(\ifibcube{\theta}{\modcatA},\mult,\eta)$ is functorial in $n$ in the following certain sense.  A surjection $s:\n{n}\to \n{m}$ induces a functor $\mathcal{P}(s): \poset{m}\to \poset{n}$ which sends $\subsetA\in\ob{\poset{m}}$ to $s^{-1}(\subsetA)$.  The functor $\mathcal{P}(s)$ is strict monoidal, i.e., preserves intersections and the empty set. In the next proposition, we show that such functors induce morphisms of the monads described in Proposition \ref{lem:Tmonad}.

\begin{proposition}\label{prop:golden} Let $(\modcatA,\we,\syshocolim)$ be a complete category equipped with a system of homotopy colimits.  If $g:\poset{m}\to \poset{n}$ is a functor that preserves intersections and the empty set, then the following statements hold. 
\begin{enumerate}
    \item The functor $g$ induces a functor 
    \begin{align*}
     g^*:\Mod{\poset{n}}&\to \Mod{\poset{m}}
     \\
     \left(\catA,\theta\right)&\mapsto \left(\catA, \theta  \left(g\times \id{\catA} \right)\right).
    \end{align*}
    
    \item For every $\poset{n}$-module $(\catA, \theta)$, there exists a morphism of monads on $\fun{\catA}{\modcatA}$
from $\left(\ifibcube{g^* \theta}{\modcatA}, \mu_{g^*\theta}, \eta_{g^*\theta}\right)$ to $\left(\ifibcube{\theta}{\modcatA}, \mu_\theta, \eta_\theta \right)$, where the unit and multiplication transformations  are as  in Proposition \ref{lem:Tmonad}.
\end{enumerate}
\end{proposition}

\begin{proof}  \textit{1.} This is a ``many objects'' version of the classical result that a ring map induces a functor between the associated module categories by restriction. More explicitly, since $g$ preserves intersections and the empty set, one can immediately verify that $g^* \theta$ satisfies the conditions required in Definition \ref{def:nmodule} to be a $\poset{n}$-module.  

Let $F:\catA\to \catB$ be a morphism of $\poset{n}$-modules from $(\catA, \theta_\catA)$ to $(\catB, \theta_\catB)$.  In the following diagram 
\begin{center} 
        \begin{tikzcd}[column sep=4.5em]
                \catA\times \poset{m}
                    \arrow[r, "\id{\catA}\times g"]    
                    \arrow[d, swap, "F\times\id{\poset{m}}\,"]
             &   
                \catA\times\poset{n}
                    \arrow[r, "\theta_{\catA}"]
                    \arrow[d, "F\times\id{\poset{n}}"]
            &   
                \catA
                    \arrow[d, "F"]
            \\
                \catB\times\poset{m} 
                    \arrow[r, swap, "\id{\catB}\times g"]
            &    
                \catB\times\poset{n}
                    \arrow[r, swap, "\theta_{\catB}"]
            &   
                \catB
        \end{tikzcd}
\end{center}
the left hand square commutes because both composites are $F\times g$, and the right hand square commutes because $F$ is a $\poset{n}$-module morphism.  Taken together, this shows that $F$ is a morphism from $\left(\catA, g^* \theta_{\catA}\right)$ to $\left(\catB, g^* \theta_{\catB}\right)$.  

\textit{2.}  To establish the second statement, we use several properties of systems of homotopy colimits.  The monad morphism from $\big(\ifibcube{g^* \theta}{\modcatA}, \mu_{g^*\theta}, \eta_{g^*\theta}\big)$ to $\left(\ifibcube{\theta}{\modcatA}, \mu_\theta, \eta_\theta \right)$ consists of the identity functor $\id{\fun{\catA}{\modcatA}}$ together with a natural transformation $\alpha:\ifibcube{g^*\theta}{\modcatA}\Rightarrow \ifibcube{\theta}{\modcatA}$. Recall that for $G\in\ob{\fun{\catA}{\modcatA}}$ and $a\in \ob\catA$, 
\begin{align*}
    \ifibcube{g^*\theta}{\modcatA}(G)(a)
        &=\hocolimunder{\Lambda^{\times m}} \ \Ran{G(g^*\theta(a, -))}{\phi_m}
        \\
    \ifibcube{\theta}{\modcatA}(G)(a)
        &=\hocolimunder{\Lambda^{\times n}} \ \Ran{G(\theta(a, -))}{\phi_n}.
\end{align*}
It is easy to see that one can explicitly construct a functor $\widehat{g}:\Lambda^{\times m}\to \Lambda^{\times n}$ making the diagram
\begin{center} 
        \begin{tikzcd}[column sep=4.5em]
                
                \poset{m}
                    \arrow[r, "g"]
                    \arrow[d, swap, "\phi_m"]
            &   
                \poset{n}
                    \arrow[d, "\phi_n"]
            \\
                
                \Lambda^{\times m}
                    \arrow[r, swap, "\widehat{g}"]
            &   
                \Lambda^{\times n}
        \end{tikzcd}
\end{center}
\begin{sloppypar}\tolerance 900
commute.  Straightforward calculation shows that the right Kan extension $\Ran{G(g^*\theta(a, -))}{\phi_m}$ is equal to $\Ran{G(\theta(a, -))}{\phi_n}  \widehat{g}$.
The natural transformation $\alpha$ is then simply the natural transformation the existence of which is guaranteed by property $2$ of Lemma \ref{lem:hocolim}. 
\end{sloppypar}

It follows from the definition of the unit natural transformation from Proposition \ref{lem:Tmonad} that for any $G\in\ob{\fun{\catA}{\modcatA}}$, the diagram
\begin{center}
        \begin{tikzcd}
                & G
                    \arrow[dl, swap, "\eta_{g^*\theta} \whisk G"]
                    \arrow[dr, "\eta_{\theta}"]
            \\  \ifibcube{g^*\theta}{\modcatA}  G
                    \arrow[rr, swap, "\alpha"]
            &&   G   \ifibcube{\theta}{\modcatA} 
        \end{tikzcd}
    \end{center}
commutes because $\theta(a, \emptyset)=\theta(a, g(\emptyset))$.

Lastly, we show that the natural transformation $\alpha$ respects the multiplication natural transformations. Using the definition of the multiplication from Proposition \ref{lem:Tmonad}, as well as of the natural transformation $\alpha$ and the monad $\ifibcube{\modcatA}{\theta}$, we see that this means that we must show that the following diagram commutes:
\begin{center}
    \begin{tikzcd}
         \hocolimunder{\Lambda^{\times n}\times \Lambda^{\times n}}\ \Ran{\cap^*G(g^*\theta(a, -))}{\phi_n\times \phi_n}
            \arrow[r]
            \arrow[d]
    &
     \hocolimunder{\Lambda^{\times n}}\ \Ran{G(g^*\theta(a, -))}{\phi_n}
            \arrow[d]
    \\
    \hocolimunder{\Lambda^{\times m}\times \Lambda^{\times m}}\ \Ran{\cap^*G(\theta(a, -))}{\phi_m\times \phi_m}
            \arrow[r]
    & \hocolimunder{\Lambda^{\times m}}\ \Ran{G(\theta(a, -))}{\phi_m}.
    \end{tikzcd}
\end{center}
Here the horizontal arrows are each instances of the multiplication map obtained using the natural transformation in Equation (\ref{eq:alpha}) and the vertical arrows are each an instance of the natural transformation $\alpha$.  In both cases, these maps are obtained by applying property $2$ of Lemma \ref{lem:hocolim} to the functors $g^*$ and $\cap^*$.  That the diagram above commutes follows immediately from the fact that $g^*$ and $\cap^*$ commute, since $g$ commutes with intersections.
\end{proof}

\begin{theorem}\label{theorem:pnmodstomonads}
Let $(\modcatA,\we,\syshocolim)$ be a complete category equipped with a system of homotopy colimits. There exists a contravariant functor
       \[
       \begin{split}
        \monadcube{n}{\modcatA}: \Mod{\poset{n}}^{\;op}&\to\monadcat\\
          (\catA,\theta) &\mapsto \monadcube{n}{\modcatA}(\theta)=\left( \fun{\catA}{\modcatA}, \left(\ifibcube{\theta}{\modcatA},\mult,\eta\right)\right).
       \end{split}
    \]
\begin{proof}
Given a $\poset{n}$-module $(\catA,\theta_\catA)$, we know from Proposition \ref{lem:Tmonad}, that $\monadcube{n}{\modcatA}(\theta_\catA)$ is indeed a monad on $\fun{\catA}{\modcatA}$, and it therefore suffices to show that a morphism of $\poset{n}$-modules $F:(\catA,\theta_\catA)\to (\catB,\theta_\catB)$ naturally induces a morphism of monads 
\[\monadcube{n}{\modcatA}\left(\theta_\catB\right)=\left(\fun{\catB}{\modcatA},\left(T_\catB, \mult_\catB,\eta_\catB\right)\right) \to \monadcube{n}{\modcatA}\left(\theta_\catA\right)=\left(\fun{\catA}{\modcatA},\left(T_\catA,\mult_\catA,\eta_\catA\right)\right).\]

We do this by setting $\monadcube{n}{\modcatA}(F)=(F^\ast,\alpha_F):\monadcube{n}{\modcatA}(\theta_\catB)\rightarrow\monadcube{n}{\modcatA}(\theta_\catA)$, where the natural transformation $\alpha_F:T_\catA  F^\ast\Rightarrow F^\ast  T_\catB$ is the identity natural transformation. This choice makes sense, since the fact that $F$ is a morphism of $\poset{n}$-modules implies that for any  $G\in\ob{\fun{\catB}{\modcatA}}$  and any $a\in\ob{\catA}$, 
\begin{align*}
    T_\catA   F^\ast (G)(a) 
        &= \hocolimunder{\Lambda^{\times n}}\Ran{\chi_{\theta_\catA, G  F}^{(-)}}{\varphi_n}(a)
    \\
        &=\hocolimunder{\Lambda^{\times n}}\Ran{G  F(\theta_\catA(a,-))}{\varphi_n}
    \\
        &= \hocolimunder{\Lambda^{\times n}}\Ran{G(\theta_\catB (F(a),-))}{\varphi_n}
    \\
        &= \left(\hocolimunder{\Lambda^{\times n}}\Ran{\chi_{\theta_\catB, G}^{(-)}}{\varphi_n}\right)  F(a)
    \\
        &=F^\ast   T_\catB (G)(a).
\end{align*}
It is easy to check that this association of  a morphism of monads to a morphism of $\poset{n}$-modules preserves composition and identities.

\
\end{proof}
   \end{theorem}
   
\section{Recovering dual calculus}

The goal of this section is to show that the cocross effects in the dual calculus tower of \cite{McCarthy2001} are obtained naturally  from a $\poset{n}$-module.  Dual calculus  is a dualization  of the constructions that appear in \cite{JM04} for functors of abelian categories (McCarthy's dual calculus was originally constructed for functors of spectra).    The functor calculus  of \cite{JM04} provides both an obstruction, the $(n+1)^{\textrm{st}}$ cross effect $\text{cr}_{n+1} F$, which measures the failure of a functor $F$ to take $(n+1)$-fold coproducts to $(n+1)$-fold products, and the best ``degree $n$'' approximation of $F$ to a functor $P_{n}F$ satisfying this property. Dually, the $(n+1)^{\textrm{st}}$ cocross effect of the dual calculus tower measures the failure of a functor to take $(n+1)$-fold products to coproducts, and can be used to provide the best ``codegree $n$'' approximation to a functor.  

The construction of the dual functor calculus tower in \cite{McCarthy2001} proceeds in stages. First, the cocross effects are defined using homotopy cofibers of certain cubical diagrams.  The cocross effects give rise to a monad, which is used to define the stages of the dual calculus tower, $P^nF$.  Finally, the functors $P^nF$ are shown to assemble into a tower of functors satisfying a universal property.  In this section, we will show that the construction of the cocross effects of a functor $F:\catA\to \catB$ can be obtained from a $\poset{n}$-module structure on the category $\catA$.  The cocross effects and the associated monad are then shown to be direct consequences of the constructions in Section 5.  

Let $\catA$ be a category with finite products and a zero object denoted by $\basepoint{\catA}$. Given a function $g:\n{n}\to \ob{\catA}$, a key element of the construction of the cocross effects in \cite[Definition 1.3]{McCarthy2001} is a functor 

\begin{align*}
   U^g: \poset{n}&\to \catA 
   \\
   \subsetB & \mapsto \prod_{v\in \subsetB} g(v).
\end{align*}
Functoriality is defined using the universal properties of $\basepoint{\catA}$ and of the product.  

Recall from Example \ref{ex:theta^n} that for each $n\geq 1$ there is a $\poset{n}$-module structure 
\[ \theta^n:\catA^{\times n} \times\poset{n}   \to \catA^{\times n} \]
taking $( (a_1, \ldots, a_n  ),\subsetB)$ to $(b_1, \ldots, b_n)$ where $b_i=a_i$ for $i\in \subsetB$ and $b_i=\basepoint{\catA}$ otherwise.
The functor $\sc{U}^g$ can now be expressed in terms of this $\poset{n}$-module structure by composing $\theta^n$ with the product functor $\sqcap$ from Example \ref{prodmonad}:
\begin{align*}
    & \catA^{\times n}\times \poset{n} \xrightarrow{\theta^n}\catA^{\times n}\xrightarrow{\sqcap} \catA.
\end{align*}
The fact that $\sqcap  \theta^n((g(1), \ldots, g(n),\subsetB)\cong U^g(\subsetB)$ follows immediately
from the definitions of $\theta^n$ and $\sqcap$ and the fact that $\basepoint{\catA}\times a\cong a$ for all objects $a\in\ob{\catA}$.  
The functoriality of $U^g$ can now be more precisely stated. For a morphism $\iota:\subsetA\subset \subsetB$ in $\poset{n}$, the morphism $U^g(\iota)$ is defined as follows: 
\[ U^g(\iota):\ = \sqcap\left(\theta^n\left(\id{(g(1), \ldots, g(n))},\iota \right) \right).\]

Throughout the rest of this section, let $(\modcatA,\we,\syshocolim)$ be a complete category equipped with a system of homotopy colimits as in Definition \ref{def:hocolim}.
Let $F:\catA\to \modcatA$ be a functor.  Using the product functor $\sqcap:\catA^{\times n}\to \catA$, we obtain an $n$-cubical diagram
\[\theintcube{}{\theta^n}{F \sqcap}:\poset{n}\to \fun{\catA^{\times n}}{\modcatA}\]
of functors, as in Equation \eqref{e:ncubefrommodule}.  Using the identification of $U^g$ with $\sqcap  \theta^n$, we restate the definition of the cocross effects from \cite{McCarthy2001}.

\begin{definition}\label{def:cocross}{\cite[Def. 1.3]{McCarthy2001}}   For a functor $F:\catA \to \modcatA$, define the \emph{$n$-th cocross effect of $F$} to be the iterated homotopy cofiber of the $n$-cube $\theintcube{}{\theta^n}{F  \sqcap}$ in $\fun{\catA^{\times n}}{\modcatA}$:
\[ \cross{n}F\ = \icofib \theintcube{}{\theta^n}{F  \sqcap}.\]
\end{definition}

In other words, using the notation of Proposition \ref{lem:Tmonad},
$$\cross{n}F=\ifibcube{\theta^n}{\modcatA}(F  \sqcap).$$

\begin{proposition}\label{thm:cocrossmonad}
   \begin{sloppypar}
       The diagonal $n$-th cocross effect $\Delta^*\cross{n}$ construction underlies a monad $(\ifibcube{}{\cross{n}}, \mu_{\cross{n}}, \eta_{\cross{n}})$  on $\fun{\catA}{\modcatA}$.  \end{sloppypar}
\end{proposition}

\begin{proof}  Let $\left(\ifibcube{\theta^n}{\modcatA}, \mu_n, \eta_n\right)$ be the monad on $\fun{\catA^{\times n}}{\modcatA}$ from Proposition \ref{lem:Tmonad} and consider the adjoint pair of functors $\Delta\dashv \sqcap:\catA\to \catA^{\times n}$ from Example \ref{prodmonad}.  Precomposition with these functors yields an adjunction $\sqcap^*\dashv \Delta^*: \fun{\catA^{\times n}}{\modcatA}\to \fun{\catA}{\modcatA}$, so by Lemma \ref{lem:adjmonad}, the composite $\ifibcube{}{\cross{n}} := \Delta^*  \ifibcube{\theta^n}{\modcatA}   \sqcap^*$ underlies a monad on $\fun{\catA}{\modcatA}$.    The functor $\ifibcube{}{\cross{n}}$ is precisely the diagonal cocross effect $\Delta^*\cross{n}$.
\end{proof}

The first main theorem of \cite{McCarthy2001} used in the construction of the dual calculus tower relates the $n$-th cocross effect to the $m$-th cocross effect of a functor to spectra whenever there is a surjection $s:\n{n}\to \n{m}$.  We now extend this to our setting.  

Let $Surj$ denote the category which has as objects the sets $\n{n}$ and surjective functions as morphisms.  It will be useful to have notation for several variants of the diagonal and product functors needed to compare $\poset{k}$-modules for different $k$'s. To keep track of these, let $\Delta^k:\catA\to \catA^{\times k}$ and $\sqcap^k:\catA^{\times k}\to \catA$ denote the relevant functors to and from the category $\catA^{\times k}$.  Given a surjection $s:\n{n}\to \n{m}$, let $\Delta^{(s)}:\catA^{\times m}\to \catA^{\times n}$ denote the diagonal functor defined by

\[ \Delta^{(s)}(a_1, \ldots, a_m)=\left(a_{s(1)},...,a_{s(n)}\right)\] 
 and let $\sqcap^{(s)}:\catA^{\times n}\to \catA^{\times m}$ denote the product functor defined by  
\[ \sqcap^{(s)} (a_1, \ldots , a_n) = \left( \prod_{i\in s^{-1}(1)} a_i\text{  }, \ldots , \prod_{i\in s^{-1}(m)} a_i \right). \]  
These two variants of the usual diagonal and product functors are in fact obtained from the usual ones as follows.  Let $|s^{-1}(i)|$ denote the cardinality of the subset $s^{-1}(i)$ in $\n{n}$.  Then the adjoint pairs of functors for $i=1, \ldots, m$ 
\[\Delta^{|s^{-1}(i)|}\dashv \sqcap^{|s^{-1}(i)|}:\catA\to \catA^{\times |s^{-1}(i)|}\]
can be combined to produce an adjoint pair of functors
\[ \Delta^{|s^{-1}(1)|}\times \cdots \times \Delta^{|s^{-1}(m)|}\dashv \sqcap^{|s^{-1}(1)|}\times \cdots \times \sqcap^{|s^{-1}(m)|}\]
from $\catA^{\times m}$ to $\catA^{s^{-1}(1)\times \cdots \times s^{-1}(m)}=\catA^{\times n}$.
One can verify from the definitions that
\begin{align*} 
\Delta^{(s)}&= \Delta^{|s^{-1}(1)|}\times \cdots \times \Delta^{|s^{-1}(m)|}
\\
\sqcap^{(s)} &= \sqcap^{|s^{-1}(1)|}\times \cdots \times \sqcap^{|s^{-1}(m)|}.
\end{align*}
In particular, there is an adjunction $\Delta^{(s)}\dashv \sqcap^{(s)}: \catA^{\times m}\to \catA^{\times n}$.

\begin{theorem}\cite[Theorem 1]{McCarthy2001} There exists a well-defined functor 
\[Surj^{op}\to \monadcat_{\fun{\catA}{\modcatA}} \]
given by sending the set $\n{n}$ to the monad $\left(\ifibcube{}{\cross{n}}, \mu_{\cross{n}}, \eta_{\cross{n}}\right)$.
\end{theorem}
\begin{sloppypar}\tolerance 900
The proof of this theorem is technical, involving many constructions that are somewhat hidden in our notation for monads arising from $\poset{n}$-modules.  The core of the proof involves producing a morphism of monads taking $\left(\ifibcube{}{\cross{m}}, \mu_{\cross{m}}, \eta_{\cross{m}}\right)$ to $\left(\ifibcube{}{\cross{n}}, \mu_{\cross{n}}, \eta_{\cross{n}}\right)$. It may be helpful for the reader to know that there are really only three key monad morphisms used in the proof.
\end{sloppypar}
\begin{enumerate}
    \item Since $\cross{n}$ and $\cross{m}$ are obtained from the $\poset{n}$-module $(\theta^n, \catA^{\times n})$ and the $\poset{m}$-module $(\theta^m, \catA^{\times m})$, respectively, there is no direct comparison of $\cross{n}$ and $\cross{m}$ via modules.  Given a surjection $s:\n{n}\to \n{m}$, we first produce a $\poset{m}$-module, $(\poset{s}^*(\theta^n)\,, \catA^{\times n})$, from the $\poset{n}$-module $(\theta^n, \catA^{\times n})$.   There is a module morphism from $(\theta^m, \catA^{\times m})$ to $(\poset{s}^*(\theta^n)\,, \catA^{\times n})$, and we apply Theorem \ref{theorem:pnmodstomonads} to obtain a map of the associated monads.
    \item The $m$-cubical diagram giving rise to the monad $\monadcube{m}{\modcatA}\big(\mathcal{P}(s)^*(\theta^n)\big)$   can be identified as a sub-cube of the $n$-cubical diagram used to compute the $n$-th cocross effect.  There is an induced map between the iterated cofibers of these cubes.
    \item The final ingredient in the proof is a relationship between the different diagonal functors needed to evaluate the $m$-th and $n$-th cocross effects on their diagonals.  The adjunction between $\Delta^{(s)}$ and $\sqcap^{(s)}$ produces the desired map of monads, using Lemma \ref{lem:conjunctionadjunction}.
\end{enumerate}

\begin{proof}
\begin{sloppypar}\tolerance 900
We proved in Proposition \ref{thm:cocrossmonad} that $\left(\ifibcube{}{\cross{n}}, \mu_{\cross{n}}, \eta_{\cross{n}}\right)$  is indeed a monad on $\fun{\catA}{\modcatA}$, so it remains to prove that for every surjection $s:\n{n}\to \n{m}$, there is a morphism of monads taking $\left(\ifibcube{}{\cross{m}}, \mu_{\cross{m}}, \eta_{\cross{m}}\right)$ to $\left(\ifibcube{}{\cross{n}}, \mu_{\cross{n}}, \eta_{\cross{n}}\right)$.  \end{sloppypar}

The surjection $s$ induces a functor $\mathcal{P}{(s)}:\poset{m}\to \poset{n}$ given by sending $\subsetA\in\ob{\poset{m}}$ to $s^{-1}(\subsetA)$.  This in turn induces a functor $\mathcal{P}(s)^*:\Mod{\poset{n}}\to \Mod{\poset{m}}$ given by sending a $\poset{n}$-module structure $\theta$  on $\catA$ to its precomposition with $\id{\catA}\times \mathcal{P}(s)$:
\[
\begin{tikzcd}
    \catA\times \poset{m} 
        \arrow[rr, "\id{\catA}\times \mathcal{P}(s)"]
    && \catA\times \poset{n}
        \arrow[r, "\theta"]
    & \catA.
\end{tikzcd}
\]
The resulting composite is a $\poset{m}$-module structure precisely because $s^{-1}$ preserves intersections and the empty set, by Proposition \ref{prop:golden}.

To recover McCarthy's result, first note that there is a commuting diagram 
\begin{center}\begin{equation*}\label{d:notmodules}
    \begin{tikzcd}
    \catA^{\times m}\times \poset{m}
        \arrow[d, swap, " \Delta^{(s)}\times \mathcal{P}(s)\,"]
        \arrow[rr, "\theta^m"]
    && \catA^{\times m}
        \arrow[d, "\Delta^{(s)}"]
    \\
    \catA^{\times n}\times \poset{n} 
        \arrow[rr, swap, "\theta^n"]
    && \catA^{\times n}.
    \end{tikzcd}
    \end{equation*}
\end{center}
Rephrasing slightly in terms of the functor $\mathcal{P}(s)^*$, we see that the functor $\Delta^{(s)}$ is a morphism of $\poset{m}$-modules from $(\catA^{\times m}, \theta^m)$ to $(\catA^{\times n}, \mathcal{P}(s)^*(\theta^n))$:
\begin{center}\begin{equation*}
    \begin{tikzcd}
    \catA^{\times m}\times \poset{m}
        \arrow[d, swap, " \Delta^{(s)}\times \id{\poset{m}}\,"]
        \arrow[rr, "\theta^m"]
    && \catA^{\times m}
        \arrow[d, "\Delta^{(s)}"]
    \\
    \catA^{\times n}\times \poset{m} 
        \arrow[rr, swap, "\mathcal{P}(s)^*(\theta^n)"]
    && \catA^{\times n}.
    \end{tikzcd}
    \end{equation*}
\end{center}
Applying Theorem \ref{theorem:pnmodstomonads} now produces a map of monads on $\fun{\catA^{\times m}}{\modcatA}$, 
\[ \monadcube{m}{\modcatA}\left(\mathcal{P}(s)^*\theta^n\right)\to\monadcube{m}{\modcatA}\left(\theta^m\right).\]

Recall from the proof of Theorem \ref{theorem:pnmodstomonads} that the monad morphism is induced by the functor
\[ \left(\Delta^{(s)}\right)^*: \fun{\catA^{\times n}}{\modcatA}\to \fun{\catA^{\times m}}{\modcatA}\]
given by precomposition with $\Delta^{(s)}$, and the natural transformation 
\[ \ifibcube{\theta^m}{\modcatA}  \left(\Delta^{{(s)}}\right)^* \Rightarrow \left(\Delta^{{(s)}}\right)^* \ifibcube{\mathcal{P}(s)^*\theta^n}{\modcatA}, \]
which is the identity natural transformation.

Now consider the monads given by pre- and post-composing with the diagonal and product functors, $\Delta^{(s)}$ and $\sqcap^{(s)}$, respectively.  Since $\Delta^{(s)}\dashv \sqcap^{(s)}$ is an adjoint pair, $\left(\sqcap^{(s)}\right)^*\dashv \left(\Delta^{(s)}\right)^*$ is also an adjoint pair, with $\left(\Delta^{(s)}\right)^*$ now being the right adjoint. Consider the diagram

\begin{center}
    \begin{tikzcd}[row sep=1.2em]
    &
    \fun{\catA^{\times n}}{\modcatA}
        \arrow[rr, swap, "\ifibcube{\mathcal{P}(s)^* \theta^n}{\modcatA}", ""{name=A}, ""{name=F1}]
        \arrow[rr, bend left=35,"\ifibcube{\theta^n}{\modcatA}", ""{name=F2}]
        \arrow[dd, "\left(\Delta^{(s)}\right)^\ast"]
    &&
    \fun{\catA^{\times n}}{\modcatA}
        \arrow[dd,  "\left(\Delta^{(s)}\right)^\ast"]
        \arrow[dr, "\left(\Delta^n\right)^\ast"]
    &
\\
    \fun{\catA}{\modcatA}
        \arrow[ur, "(\sqcap^n)^\ast"]
        \arrow[dr, swap, "\left(\sqcap^m\right)^\ast", ""{name=alpha}]
    &
    &
    &
    &
    \fun{\catA}{\modcatA}.
\\
    &
    \fun{\catA^{\times m}}{\modcatA}
        \arrow[rr, swap, "\ifibcube{\theta^m}{\modcatA}"]
    &&
    \fun{\catA^{\times m}}{\modcatA}
        \arrow[ur, swap,  "\left(\Delta^m\right)^\ast"]
    &
    \arrow[Rightarrow, swap, "\alpha", shorten=5mm, from=alpha, to=1-2]
    \arrow[Rightarrow, "\eta\;", shorten=2mm, from=F1, to=F2]
    \end{tikzcd}
\end{center}
The existence of a morphism of monads
\[ \eta: \left(\ifibcube{\mathcal{P}(s)^*\theta^n}{\modcatA}, \mu_s, \eta_s\right) \to \left(\ifibcube{\theta^n}{\modcatA}, \mu, \eta\right)\]
is a consequence  of Proposition \ref{prop:golden}. The existence of the natural transformation $\alpha$ follows by noting that 
\[ \sqcap^n = \sqcap^m  \sqcap^{(s)},\]
so that $\left(\Delta^{(s)}\right)^*  \left(\sqcap^n\right)^* = \left(\Delta^{(s)}\right)^*  \left(\sqcap^{(s)}\right)^*  \left(\sqcap^m\right)^*$.
Whiskering the unit 
$$u^{(s)}:\id{\fun{\catA^{\times m}}{\modcatA}}\Rightarrow \left(\Delta^{(s)}\right)^*  \left(\sqcap^{(s)}\right)^*$$ 
of the adjunction $\left(\sqcap^{(s)}\right)^*\dashv \left(\Delta^{(s)}\right)^*$ on the right with $\left(\sqcap^m\right)^*$ gives rise to a natural transformation 
$$\alpha=u^{(s)}*\left(\sqcap^m\right)^*:\left(\sqcap^m\right)^*\Rightarrow \left(\Delta^{(s)}\right)^*  \left(\sqcap^n\right)^*.$$ 
We are now in the context of Lemma \ref{lem:conjunctionadjunction}, with
$$L= (\sqcap^n)^*,\quad R= (\Delta^n)^*,\quad T= \ifibcube{\modcatA}{\mathcal{P}(s)^* \theta^n},$$
$$L'= \left(\sqcap^m\right)^*,\quad R'= (\Delta^m)^*,\quad T'= \ifibcube{\modcatA}{\theta^m},$$
and
$$F=\left(\Delta ^{(s)}\right)^*,\quad \tau_L=\alpha, \quad\tau_R=\id{(\Delta^{n})^{*}}.$$
To ensure that we can apply Lemma \ref{lem:conjunctionadjunction}, we need to show that
$$\left(\left(\Delta ^{(s)}\right)^*\whisk\epsilon^n\right)\left(\alpha \whisk\left(\Delta^n\right)^*\right) = \epsilon^m\whisk(\Delta ^{(s)})^*$$
and
$$\left((\Delta^m)^* \whisk\alpha\right)u^m =u^n,$$
where $\epsilon ^k$ and $u^k$ denotes the counit and unit of the adjunction $(\sqcap^{k})^*\dashv (\Delta^{k})^*$ for all $k$.  The second of these equations is an immediate consequence of the fact that the adjunction $(\sqcap^{n})^*\dashv (\Delta^{n})^*$ is the composite $(\sqcap^{m})^*\dashv (\Delta^{m})^*$ and $\big(\sqcap^{(s)}\big)^*\dashv \big(\Delta^{(s)}\big)^*$, while the proof of the second requires further that we apply the triangle identity for the counit of the adjunction $\big(\sqcap^{(s)}\big)^*\dashv \big(\Delta^{(s)}\big)^*$.

Thus there is a morphism of the monads whose underlying functors are the composites along the bottom and top of the diagram above:
\[ (\Delta^m)^*  \ifibcube{\modcatA}{\theta^m}  (\sqcap^m)^* \Rightarrow (\Delta^n)^*  \ifibcube{\modcatA}{\mathcal{P}(s)^* \theta^n}  (\sqcap^n)^*. \]
The end result, composing this monad morphism with $\eta$,  is a natural transformation 
\[
\ifibcube{}{\cross{m}} = (\Delta^m)^*  \ifibcube{\theta^m}{\modcatA}  (\sqcap^m)^* \Rightarrow (\Delta^n)^*  \ifibcube{\theta^n}{\modcatA}  (\sqcap^n)^* = \ifibcube{}{\cross{n}}
\]
underlying a monad morphism from  $\left(\ifibcube{}{\cross{m}}, \mu_{\cross{m}}, \eta_{\cross{m}}\right)$ to  $\left(\ifibcube{}{\cross{n}}, \mu_{\cross{n}}, \eta_{\cross{n}}\right)$ 
as desired.

\end{proof}

\bibliography{cmcdhc}{}
\bibliographystyle{plain}

\end{document}
